\documentclass{siamart0516}
\usepackage{amsfonts}
\usepackage{accents}
\usepackage[margin=1in, letterpaper]{geometry}
\usepackage{graphicx}
\usepackage{bmpsize}
\usepackage{subcaption}
\usepackage{float}
\usepackage{epsfig}
\usepackage{multicol}
\usepackage{cite}
\usepackage{adjustbox,lipsum}

\newcommand{\vertiii}[1]{{\left\vert\kern-0.25ex\left\vert\kern-0.25ex\left\vert #1\right\vert\kern-0.25ex\right\vert\kern-0.25ex\right\vert}}

\begin{document}

\title{Ensemble Timestepping Algorithms for Natural Convection}
\author{J. A. Fiordilino\thanks{University of Pittsburgh, Department of Mathematics, Pittsburgh, PA 15260} \and S. Khankan\footnotemark[1]}
\maketitle
%%%%%%%%%%%%%%%%

\begin{abstract}
This paper presents two algorithms for calculating an ensemble of solutions to laminar natural convection problems.  The ensemble average is the most likely temperature distribution and its variance gives an estimate of prediction reliability.  Solutions are calculated by solving two coupled linear systems, each involving a shared coefficient matrix, for multiple right-hand sides at each timestep.  Storage requirements and computational costs to solve the system are thereby reduced.  Stability and convergence of the method are proven under a timestep condition involving fluctuations.  A series of numerical tests, including predictability horizons, are provided which confirm the theoretical analyses and illustrate uses of ensemble simulations.
\end{abstract}

\section{Introduction}
Ensemble calculations are essential in predictions of the most likely outcome of systems with uncertain data, e.g., weather forecasting \cite{Kalnay}, ocean modeling \cite{Lermusiaux}, turbulence \cite{Nan2}, etc.  Ensemble simulations classically involve J sequential, fine mesh runs or J parallel, coarse mesh runs of a given code.  This leads to a competition between ensemble size and mesh density.  We develop linearly implicit timestepping methods with shared coefficient matrices to address this issue.  For such methods, it is more efficient in both storage and solution time to solve J linear systems with a shared coefficient matrix than with J different matrices.  

Prediction of thermal profiles is essential in many applications \cite{Bairi,Gebhart,Ostrach,Marshall}.  Herein, we extend \cite{Nan} from isothermal flows to temperature dependent natural convection.  We consider two natural convection problems enclosed in mediums with: \textbf{non-zero wall thickness} \cite{Boland} and \textbf{zero wall thickness}; Figure 1 illustrates a typical setup.  The latter problem is often utilized as a thin wall approximation.

Consider the \textbf{Thick wall problem}.  Let $\Omega_{f} \subset \Omega $ be polyhedral domains in $\mathbb{R}^d (d=2,3)$ with boundaries $\partial \Omega_{f}$ and $\partial \Omega$, respectively, such that dist($\partial \Omega_{f}$,$\partial \Omega$) $> 0$.  The boundary $\partial \Omega$ is partitioned such that $\partial \Omega = \Gamma_{1}  \cup \Gamma_{2}$ with $\Gamma_{1} \cap \Gamma_{2} =\emptyset$ and $|\Gamma_{1}| > 0$.  Given $u(x,0;\omega_{j}) = u^{0}(x;\omega_{j})$ and $T(x,0;\omega_{j}) = T^{0}(x;\omega_{j})$ for $j = 1,2, ... , J$, let $u(x,t;\omega_{j}):\Omega \times (0,t^{\ast}] \rightarrow \mathbb{R}^{d}$, $p(x,t;\omega_{j}):\Omega \times (0,t^{\ast}] \rightarrow \mathbb{R}$, and $T(x,t;\omega_{j}):\Omega \times (0,t^{\ast}] \rightarrow \mathbb{R}$ satisfy
\begin{align}
u_{t} + u \cdot \nabla u -Pr \Delta u + \nabla p &= PrRa\gamma T + f \; \; in \; \Omega_{f}, \label{s1} \\
\nabla \cdot u &= 0 \; \; in \; \Omega_{f},  \\
T_{t} + u \cdot \nabla T - \nabla \cdot (\kappa \nabla T) &= g \; \; in \; \Omega, \label{s1T}  \\
u = 0 \; \; on \; \partial \Omega_{f}, \; \;
%u = 0 \; \;in \; \Omega - \Omega_{f} \times (0,t^{\ast}],  \\
%T = 0 \; \; on \; \Gamma_{2} \times (0,t^{\ast}]  \; \; and \; \;
%n \cdot \nabla T = 0 \; \; on \; \Gamma_{1} \times (0,t^{\ast}],
u = 0 \; \;in \; \Omega - \Omega_{f},  \; \;
T &= 0 \; \; on \; \Gamma_{1} \; \; and \; \;
n \cdot \nabla T = 0 \; \; on \; \Gamma_{2} \label{s1f}.
\end{align}
\noindent Here $n$ denotes the usual outward normal, $\gamma$ denotes the unit vector in the direction of gravity, $Pr$ is the Prandtl number, $Ra$ is the Rayleigh number, and $\kappa = \kappa_{f}$ in $\Omega_{f}$ and $\kappa = \kappa_{s}$ in $\Omega - \Omega_{f}$ is the thermal conductivity of the fluid or solid medium.  Further, $f$ and $g$ are the body force and heat source, respectively. 

Let $<u>^{n} := \frac{1}{J} \sum_{j=1}^{J} u^{n}$ and ${u'}^{n} = u^{n} - <u>^{n}$.  To present the idea, suppress the spatial discretization for the moment.  We apply an implicit-explicit time-discretization to the system (\ref{s1}) - (\ref{s1f}), while keeping the coefficient matrix independent of the ensemble members.  This leads to the following timestepping method:
\begin{align} 
\frac{u^{n+1} - u^{n}}{\Delta t} + <u>^{n} \cdot \nabla u^{n+1} + {u'}^{n} \cdot \nabla u^{n} - Pr \triangle u^{n+1} + \nabla p^{n+1} &= PrRa\gamma T^{n+1} + f^{n+1}, \label{d1a} \\ 
\nabla \cdot u^{n+1} &= 0, \\
\frac{T^{n+1} - T^{n}}{\Delta t} + <u>^{n} \cdot \nabla T^{n+1} + {u'}^{n} \cdot \nabla T^{n} - \kappa \Delta T^{n+1} &= g^{n+1}, \label{d2a}
\end{align}
Consider the \textbf{Thin wall problem}.  The main difference is a ``$u_{1}$" term on the r.h.s of the temperature equation (\ref{s2T}) absent in (\ref{s1T}).  This apparently small difference in the model produces a significant difference in the stability of the approximate solution.  In particular, a discrete Gronwall inequality is used which allows for the loss of long-time stability; see Section 4 below.  Consider:
\begin{align}
u_{t} + u \cdot \nabla u -Pr \Delta u + \nabla p &= PrRa\gamma T + f \; \; in \; \Omega, \label{s2} \\
\nabla \cdot u &= 0 \; \; in \; \Omega,  \\
T_{t} + u \cdot \nabla T - \nabla \cdot (\kappa \nabla T) &= u_{1} + g \; \; in \; \Omega, \label{s2T} \\
u = 0 \; \; on \; \partial \Omega,  \; \; \;
T = 0 \; \; on \; \Gamma_{1}, \; \; \;
n \cdot \nabla T &= 0 \; \; on \; \Gamma_{2}, \label{s2f}
\end{align}
\noindent where $u_{1}$ is the first component of the velocity.  If we again momentarily disregard the spatial discretization, our timestepping method can be written as:
\begin{align} 
\frac{u^{n+1} - u^{n}}{\Delta t} + <u>^{n} \cdot \nabla u^{n+1} + {u'}^{n} \cdot \nabla u^{n} - Pr \triangle u^{n+1} + \nabla p^{n+1} &= PrRa\gamma T^{n} + f^{n+1}, \label{d1}\\
\nabla \cdot u^{n+1} &= 0, \\ 
\frac{T^{n+1} - T^{n}}{\Delta t} + <u>^{n} \cdot \nabla T^{n+1} + {u'}^{n} \cdot \nabla T^{n} - \kappa \Delta T^{n+1} &= u^{n}_{1} + g^{n+1}\label{d2},
\end{align}
By lagging both $u'$ and the coupling terms in the method, the fluid and thermal problems uncouple and each sub-problem has a shared coefficient matrix for all ensemble members.

\noindent \textbf{Remark:}  The formulation (\ref{d1a}) - (\ref{d2a}) arises, e.g., in the study of natural convection within a unit square or cubic enclosure with a pair of differentially heated vertical walls.  In particular, the temperature distribution is decomposed into $\theta (x,t) = T(x,t) + \phi(x)$, where $\phi(x) = 1 - x_{1}$ is the linear conduction profile and $T(x,t)$ satisfies homogeneous boundary conditions on the corresponding pair of vertical walls.

In Section 2, we collect necessary mathematical tools.  In Section 3, we present algorithms based on (\ref{d1a}) - (\ref{d2a}) and (\ref{d1}) - (\ref{d2}).  Stability and error analyses follow in Section 4.  We end with numerical experiments and conclusions in Sections 5 and 6.  In particular, two stable, convergent ensemble algorithms are presented.  These algorithms can be used to efficiently compute an ensemble of solutions to (\ref{s1}) - (\ref{s1f}) and (\ref{s2}) - (\ref{s2f}) and estimate predictability horizons.  The ensemble average is shown to produce a better estimate of the energy in the system, for a test problem, than any member of the ensemble.
\begin{figure}
	\centering
	\includegraphics[width=5.5in,height=\textheight, keepaspectratio]{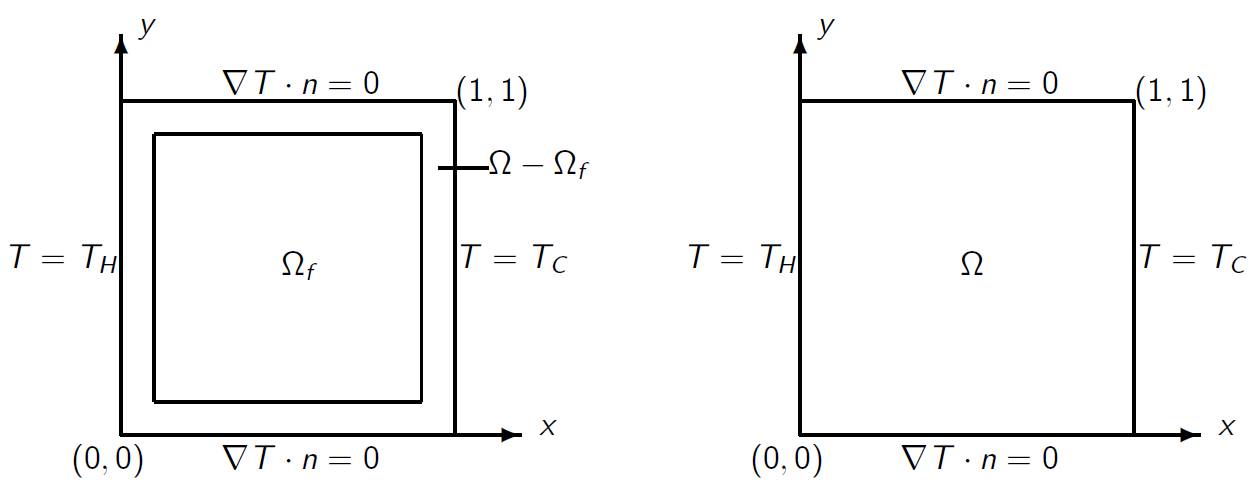}
	\caption{Domain and boundary conditions for (a) thick walled (b) thin walled double pane window problem benchmark.}
\end{figure}
%%%%%%%%%%%%%%%%%%%%%%%%%%%%%%%%%%%%%%
%%%%%%%%%%%%%%%%%%%%%%%%%%%%%%%%%%%%%%
\section{Mathematical Preliminaries}
The $L^{2} (\Omega)$ inner product is $(\cdot , \cdot)$ and the induced norm is $\| \cdot \|$.  Define the Hilbert spaces,
\begin{align*}
X &:= H^{1}_{0}(\Omega)^{d} = \{ v \in H^{1}(\Omega)^d : v = 0 \; on \; \partial \Omega \}, \;
Q := L^{2}_{0}(\Omega) = \{ q \in L^{2}(\Omega) : \int_{\Omega} q dx = 0 \}, \\
W &:= \{ S \in H^{1}(\Omega) : S = 0 \; on \; \Gamma_{1} \}, \;
V := \{ v \in X : (q,\nabla \cdot v) = 0 \; \forall q \in Q \}.
\end{align*}
The explicitly skew-symmetric trilinear forms are denoted:
\begin{align*}
b(u,v,w) &= \frac{1}{2} (u \cdot \nabla v, w) - \frac{1}{2} (u \cdot \nabla w, v) \; \; \; \forall u,v,w \in X, \\
b^{\ast}(u,T,S) &= \frac{1}{2} (u \cdot \nabla T, S) - \frac{1}{2} (u \cdot \nabla S, T) \; \; \; \forall u \in X, \; \forall T,S \in W.
\end{align*}
\noindent They enjoy the following continuity results and properties.
\begin{lemma} \label{l1}
There are constants $C_{1}, C_{2}, C_{3}, C_{4}, C_{5},$ and $C_{6}$ such that for all u,v,w $\in$ X and T,S $\in W$, $b(u,v,w)$ and $b^{\ast}(u,T,S)$ satisfy
\begin{align*}
b(u,v,w) &= \int_{\Omega} u \cdot \nabla v \cdot w dx + \frac{1}{2} \int_{\Omega} (\nabla \cdot u)v \cdot w dx, \\
b^{\ast}(u,T,S) &= \int_{\Omega} u \cdot \nabla T S dx + \frac{1}{2} \int_{\Omega} (\nabla \cdot u)T S dx, \\
b(u,v,w) &\leq C_{1} \| \nabla u \| \| \nabla v \| \| \nabla w \|, \\
b(u,v,w) &\leq C_{2} \sqrt{\| u \| \| \nabla u \|} \| \nabla v \| \| \nabla w \|, \\
b^{\ast}(u,T,S) &\leq C_{3} \| \nabla u \| \| \nabla T \| \| \nabla S \|,\\
b^{\ast}(u,T,S) &\leq C_{4} \sqrt{\| u \| \| \nabla u \|} \| \nabla T \| \| \nabla S \|,\\
b(u,v,w) &\leq C_{5} \| \nabla u \| \| \nabla v \| \sqrt{\| w \| \| \nabla w \|}, \\
b^{\ast}(u,T,S) &\leq C_{6} \| \nabla u \| \| \nabla T \| \sqrt{\| S \| \| \nabla S \|}.
\end{align*}
\begin{proof}
The proof of the first two identities is a calculation.  The next four results follow from applications of H\"{o}lder and Sobolev embedding inequalities; see Lemma 2.2 on p. 2044 of \cite{Layton}.  We will prove the last two results for $d = 3 $; for $d = 2$ they are improvable.  For all u,v,w $\in$ X,
\begin{align*}
|(u \cdot \nabla v, w)| &\leq C \| u \|_{L^{6}} \| \nabla v \| \| w \|_{L^{3}} 
\\ &\leq C \| \nabla u \| \| \nabla v \| \sqrt{\| w \| \| \nabla w \|},
\end{align*}
where H\"{o}lder, Ladyzhenskaya and Gagliardo-Nirenberg inequalities were used, respectively.  Using the above result and inequalities and the first identity in Lemma \ref{l1},
\begin{align*}
|b(u,v,w)| &= |(u \cdot \nabla v, w) + \frac{1}{2} \int_{\Omega} (\nabla \cdot u) v \cdot w dx | 
\\ &\leq |(u \cdot \nabla v, w)| + |\frac{1}{2} \int_{\Omega} (\nabla \cdot u) v \cdot w dx | 
\\ &\leq C \| \nabla u \| \| \nabla v \| \sqrt{\| w \| \| \nabla w \|} + C \| \nabla \cdot u \| \| v \|_{L^{6}} \| w \|_{L^{3}}
\\ &\leq C \| \nabla u \| \| \nabla v \| \sqrt{\| w \| \| \nabla w \|} + C \| \nabla u \| \| \nabla v \| \sqrt{\| w \| \| \nabla w \|} 
\\ &\leq C \| \nabla u \| \| \nabla v \| \sqrt{\| w \| \| \nabla w \|}.
\end{align*}
In similar fashion, there is a $C = C(\Omega)$ such that
\begin{align*}
|b^{\ast}(u,T,S)| &\leq |(u \cdot \nabla T, S)| + |\frac{1}{2} \int_{\Omega} (\nabla \cdot u) T S dx | 
\\ &\leq C \| \nabla u \| \| \nabla T \| \sqrt{\| S \| \| \nabla S \|} + C \| \nabla \cdot u \| \| T \|_{L^{6}} \| S \|_{L^{3}}
\\ &\leq C \| \nabla u \| \| \nabla T \| \sqrt{\| S \| \| \nabla S \|} + C \| \nabla u \| \| \nabla T \| \sqrt{\| S \| \| \nabla S \|}
\\ &\leq C \| \nabla u \| \| \nabla T \| \sqrt{\| S \| \| \nabla S \|}.
\end{align*}
\end{proof}
\end{lemma}
The weak formulation of system (\ref{s1}) - (\ref{s1f}) is: Find $u:[0,t^{\ast}] \rightarrow X$, $p:[0,t^{\ast}] \rightarrow Q$, $T:[0,t^{\ast}] \rightarrow W$ for a.e. $t \in (0,t^{\ast}]$ satisfying for $j = 1,...,J$:
\begin{align}
(u_{t},v) + b(u_,u,v) + Pr(\nabla u,\nabla v) - (p, \nabla \cdot v) &= PrRa(\gamma T,v) + (f,v) \; \; \forall v \in X, \\
(q, \nabla \cdot u) &= 0 \; \; \forall q \in Q, \\
(T_{t},S) + b^{\ast}(u,T,S) + \kappa (\nabla T,\nabla S) &= (g,S) \; \; \forall S \in W.
\end{align}
Similarly, the weak formulation of system (\ref{s2}) - (\ref{s2f}) is:
Find $u:[0,t^{\ast}] \rightarrow X$, $p:[0,t^{\ast}] \rightarrow Q$, $T:[0,t^{\ast}] \rightarrow W$ for a.e. $t \in (0,t^{\ast}]$ satisfying for $j = 1,...,J$:
\begin{align}
(u_{t},v) + b(u_,u,v) + Pr(\nabla u,\nabla v) - (p, \nabla \cdot v) &= PrRa(\gamma T,v) + (f,v) \; \; \forall v \in X, \\
(q, \nabla \cdot u) &= 0 \; \; \forall q \in Q, \\
(T_{t},S) + b^{\ast}(u,T,S) + \kappa (\nabla T,\nabla S) &= (u_{1},S) + (g,S) \; \; \forall S \in W.
\end{align}
\subsection{Finite Element Preliminaries}
Consider a regular, quasi-uniform mesh $\Omega_{h} = \{K\}$ of $\Omega$ with maximum triangle diameter length $h$.  Further, for the system (\ref{s1}) - (\ref{s1f}), suppose that $\partial \Omega_{f}$ and $\partial \Omega - \partial \Omega_{f}$ lie along the meshlines of the triangulation of $\Omega$.  Let $X_{h} \subset X$, $Q_{h} \subset Q$, and $W_{h} \subset W$ be conforming finite element spaces consisting of continuous piecewise polynomials of degrees \textit{j}, \textit{l}, and \textit{j}, respectively.  Moreover, assume they satisfy the following approximation properties $\forall 1 \leq j,l \leq k,m$:
\begin{align}
\inf_{v_{h} \in X_{h}} \Big\{ \| u - v_{h} \| + h\| \nabla (u - v_{h}) \| \Big\} &\leq Ch^{k+1} \lvert u \rvert_{k+1}, \label{a1}\\
\inf_{q_{h} \in Q_{h}}  \| p - q_{h} \| &\leq Ch^{m} \lvert p \rvert_{m}, \label{a2}\\
\inf_{S_{h} \in W_{h}}  \Big\{ \| T - S_{h} \| + h\| \nabla (T - S_{h}) \| \Big\} &\leq Ch^{k+1} \lvert T \rvert_{k+1}, \label{a3}
\end{align}
for all $u \in X \cap H^{k+1}(\Omega)^{d}$, $p \in Q \cap H^{m}(\Omega)$, and $T \in W \cap H^{k+1}(\Omega)$.  Furthermore, we consider those spaces for which the discrete inf-sup condition is satisfied,
\begin{equation} \label{infsup} 
\inf_{q_{h} \in Q_{h}} \sup_{v_{h} \in X_{h}} \frac{(q_{h}, \nabla \cdot v_{h})}{\| q_{h} \| \| \nabla v_{h} \|} \geq \beta > 0,
\end{equation}
\noindent where $\beta$ is independent of $h$.  The space of discretely divergence free functions is defined by 
\begin{align*}
V_{h} := \{v_{h} \in X_{h} : (q_{h}, \nabla \cdot v_{h}) = 0, \forall q_{h} \in Q_{h}\}.
\end{align*}
The space $V_{h}^{\ast}$, dual to $V_{h}$, is endowed with the following dual norm
\begin{align*}
	\| w \|_{V_{h}^{\ast}} := \sup_{v_{h} \in V_{h}}\frac{(w,v_{h})}{\|\nabla v_{h}\|}.
\end{align*}
The discrete inf-sup condition implies that we may approximate functions in $V$ well by functions in $V_{h}$,
\begin{lemma} \label{l5}
	Suppose the discrete inf-sup condition (\ref{infsup}) holds, then for any $v \in V$
	\begin{equation*}
	\inf_{v_{h} \in V_{h}} \| \nabla (v - v_{h}) \| \leq C(\beta)\inf_{v_{h} \in X_{h}} \| \nabla (v - v_{h}) \|.
	\end{equation*}
\end{lemma}
\begin{proof}
	See Chapter 2, Theorem 1.1 on p. 59 of \cite{Girault}.
\end{proof}
We will also assume that the mesh and finite element spaces satisfy the standard inverse inequality \cite{Ern}:
$$\| \nabla \chi_{1,2} \| \leq C_{inv,1,2}(\alpha_{min}) h^{-1} \| \chi_{1,2} \| \; \; \; \forall \chi_{1} \in X_{h}, \; \forall \chi_{2} \in W_{h},$$ 
\noindent where $\alpha_{min}$ denotes the minimum angle in the triangulation.  A discrete Gronwall inequality will play a role in the upcoming analysis.
\begin{lemma} \label{l4}
(Discrete Gronwall Lemma). Let $\Delta t$, H, $a_{n}$, $b_{n}$, $c_{n}$, and $d_{n}$ be finite nonnegative numbers for n $\geq$ 0 such that for N $\geq$ 1
\begin{align*}
a_{N} + \Delta t \sum^{N}_{0}b_{n} &\leq \Delta t \sum^{N-1}_{0} d_{n}a_{n} + \Delta t \sum^{N}_{0} c_{n} + H,
\end{align*}
then for all  $\Delta t > 0$ and N $\geq$ 1
\begin{align*}
a_{N} + \Delta t \sum^{N}_{0}b_{n} &\leq exp\big(\Delta t \sum^{N-1}_{0} d_{n}\big)\big(\Delta t \sum^{N}_{0} c_{n} + H\big).
\end{align*}
\end{lemma}
\begin{proof}
	See Lemma 5.1 on p. 369 of \cite{Heywood}.
\end{proof}
The discrete time analysis will utilize the following norms $\forall \; 1 \leq k \leq \infty$:
\begin{align*}
\vertiii{v}_{\infty,k} &:= \max_{0\leq n \leq N} \| v^{n} \|_{k}, \;
\vertiii{v}_{p,k} := \big(\Delta t \sum^{N}_{n = 0} \| v^{n} \|^{p}_{k}\big)^{1/p}.
\end{align*}
%%%%%%%%%%%%%%%%%%%%%%%%%%%%%%%%%%%%%%%%
%%%%%%%%%%%%%%%%%%%%%%%%%%%%%%%%%%%%%%%%
\section{Numerical Scheme}
Denote the fully discrete solutions by $u^{n}_{h}$, $p^{n}_{h}$, and $T^{n}_{h}$ at time levels $t^{n} = n\Delta t$, $n = 0,1,...,N$, and $t^{\ast}=N\Delta t$.  Given $(u^{n}_{h}, p^{n}_{h}, T^{n}_{h})$ $\in (X_{h},Q_{h},W_{h})$, find $(u^{n+1}_{h}, p^{n+1}_{h}, T^{n+1}_{h})$ $\in (X_{h},Q_{h},W_{h})$ satisfying, for every $n = 0,1,...,N$, the fully discrete approximation of the \textbf{Thick wall problem:}
\begin{multline}\label{scheme:one:velocity}
(\frac{u^{n+1}_{h} - u^{n}_{h}}{\Delta t},v_{h}) + b(<u_{h}>^{n},u^{n+1}_{h},v_{h}) + b({u'}^{n}_{h},u^{n}_{h},v_{h}) + Pr(\nabla u^{n+1}_{h},\nabla v_{h}) - (p^{n+1}_{h}, \nabla \cdot v_{h})
\\  =  PrRa(\gamma T^{n+1}_{h},v_{h}) + (f^{n+1},v_{h}) \; \; \forall v_{h} \in X_{h},
\end{multline}
\begin{align}
(q_{h}, \nabla \cdot u^{n+1}_{h}) = 0 \; \; \forall q_{h} \in Q_{h},
\end{align}
\begin{multline}\label{scheme:one:temperature}
(\frac{T^{n+1}_{h} - T^{n}_{h}}{\Delta t},S_{h}) + b^{\ast}(<u_{h}>^{n},T^{n+1}_{h},S_{h}) + b^{\ast}({u'}^{n}_{h},T^{n}_{h},S_{h}) + \kappa (\nabla T^{n+1}_{h},\nabla S_{h})  
\\ = (g^{n+1},S_{h}) \; \; \forall S_{h} \in W_{h}.
\end{multline}

\noindent \textbf{Thin wall problem:}
\begin{multline}\label{scheme:two:velocity}
(\frac{u^{n+1}_{h} - u^{n}_{h}}{\Delta t},v_{h}) + b(<u_{h}>^{n},u^{n+1}_{h},v_{h}) + b({u'}^{n}_{h},u^{n}_{h},v_{h}) + Pr(\nabla u^{n+1}_{h},\nabla v_{h}) - (p^{n+1}_{h}, \nabla \cdot v_{h}) 
\\ =  PrRa(\gamma T^{n}_{h},v_{h}) + (f^{n+1},v_{h}) \; \; \forall v_{h} \in X_{h},
\end{multline}
\begin{equation}
(q_{h}, \nabla \cdot u^{n+1}_{h}) = 0 \; \; \forall q_{h} \in Q_{h},
\end{equation}
\begin{multline}\label{scheme:two:temperature}
(\frac{T^{n+1}_{h} - T^{n}_{h}}{\Delta t},S_{h}) + b^{\ast}(<u_{h}>^{n},T^{n+1}_{h},S_{h}) + b^{\ast}({u'}^{n}_{h},T^{n}_{h},S_{h}) + \kappa (\nabla T^{n+1}_{h},\nabla S_{h})  
\\ = (u^{n}_{1},S_{h}) + (g^{n+1},S_{h}) \; \; \forall S_{h} \in W_{h}.
\end{multline}
\textbf{Remark:} The treatment of the nonlinear terms in the time discretizations (\ref{d1a}) - (\ref{d2a}) and (\ref{d1}) - (\ref{d2}) leads to a shared coefficient matrix independent of the ensemble members.
\section{Numerical Analysis of the Ensemble Algorithm}
We present stability results for the aforementioned algorithms under the following timestep condition:
\begin{align}
\frac{C_{\dagger} \Delta t}{h} \max_{1 \leq j \leq J} \|\nabla {u'}^{n}_{h}\|^{2} \leq 1,\label{c1}
\end{align}
\noindent where $C_{\dagger} \equiv C_{\dagger}(|\Omega|, \alpha_{min},\kappa,Pr)$.
In Theorems \ref{t1} and \ref{t2}, the nonlinear stability of the velocity, temperature, and pressure approximations are proven under condition (\ref{c1}) for the thick wall (\ref{scheme:one:velocity}) - (\ref{scheme:one:temperature}) and thin wall problems (\ref{scheme:two:velocity}) - (\ref{scheme:two:temperature}), respectively.  

\noindent \textbf{Remark:}  Stability of the numerical approximations can also be proven under:
$\frac{J C_{\dagger} \Delta t}{h} < \|\nabla {u'}^{n}_{h}\|^{2} > \leq 1$.  If $C_{\dagger}/J \geq 1$, then $JC_{\dagger}$ can be replaced with $C_{\dagger}$.
%We will begin by proving that the temperature is stable for all $N \geq 1$ under the condition: 
%$$\frac{C_{3}^{\ast} \Delta t}{\kappa h} \|\nabla u^{n}_{h}\|^{2} \leq 1$$
%for $j = 1,...,J$.  After gathering stability of the temperature, we use this established result to yield stability of the velocity under the condition: 
%$$\frac{C^{\ast} \Delta t}{Pr h} \|\nabla u^{n}_{h}\|^{2} \leq 1$$
% for $j = 1,...,J$.  We then use the stability of the temperature and velocity to prove stability of the pressure under the discrete inf-sup condition.
%\noindent Note: Denote $C1 = \frac{C_{3}^{\ast} \Delta t}{\kappa h}$ and $C2 = \frac{C^{\ast} \Delta t}{Pr h}$. Take $\max \{C1,C2\} \|\nabla u^{n}_{h}\|^{2} \leq 1$ as our condition for the Theorem.

\subsection{Stability Analysis}

\begin{theorem} \label{t1}
Consider the \textbf{Thick wall problem} (\ref{scheme:one:velocity}) - (\ref{scheme:one:temperature}).   Suppose $f \in L^{\infty}(0,t^{\ast};H^{-1}(\Omega)^{d})$, $g \in L^{\infty}(0,t^{\ast};H^{-1}(\Omega))$.  If (\ref{scheme:one:velocity}) - (\ref{scheme:one:temperature}) satisfy condition (\ref{c1}), then
\begin{multline*}
\|T^{N}_{h}\|^{2} + \|u^{N}_{h}\|^{2} + \frac{1}{2}\sum_{n = 0}^{N-1}\big(\|T^{n+1}_{h} - T^{n}_{h}\|^{2} + \|u^{n+1}_{h} - u^{n}_{h}\|^{2}\big) + {\kappa \Delta t} \|\nabla T^{N}_{h}\|^{2} + {Pr \Delta t} \|\nabla u^{N}_{h}\|^{2} 
\\ \leq 2 \Delta t Pr Ra^{2}C_{PF,1}^2\sum^{N-1}_{n=0}\big(\frac{\Delta t}{\kappa} \sum_{k = 0}^{n} \|g^{k+1}\|^{2}_{-1} + \|T^{0}_{h}\|^{2} + {\kappa \Delta t} \|\nabla T^{0}_{h}\|^{2}\big) + \frac{2 \Delta t}{Pr} \sum_{n=0}^{N-1} \|f^{n+1} \|^{2}_{-1} + \|u^{0}_{h}\|^{2}
\\ + {Pr \Delta t} \|\nabla u^{0}_{h}\|^{2} + \|T^{0}_{h}\|^{2} + {\kappa \Delta t} \|\nabla T^{0}_{h}\|^{2}.
\end{multline*}
\noindent Further,
\begin{multline*}
\beta \Delta t \sum^{N-1}_{n=0} \| p^{n+1}_{h}\| \leq 2 \Delta t \sum^{N-1}_{n=0} \Big( C_{1} \| \nabla <u_{h}>^{n} \| \| \nabla u^{n+1}_{h} \| + C_{1} \| \nabla  {u'}^{n}_{h} \|\|\nabla u^{n}_{h} \| 
\\ + Pr \|\nabla u^{n+1}_{h} \| +  PrRaC_{PF,1} \|T^{n+1}_{h} \| + \|f^{n+1} \|_{-1}\Big).
\end{multline*}
\end{theorem}
\begin{proof}
Let $S_{h} = T^{n+1}_{h}$ in equation (\ref{scheme:one:temperature}) and use the polarization identity.  Multiply by $\Delta t$ on both sides and rearrange.  Then,
\begin{align}
\frac{1}{2} \Big\{\|T^{n+1}_{h}\|^{2} - \|T^{n}_{h}\|^{2} + \|T^{n+1}_{h} - T^{n}_{h}\|^{2}\Big\} + \kappa  \Delta t \|\nabla T^{n+1}_{h}\|^{2} =  \Delta t (g^{n+1},T^{n+1}_{h}) - \Delta t b^{\ast}({u'}^{n}_{h},T^{n}_{h},T^{n+1}_{h}). \label{stability:thick}
\end{align}
Use Cauchy-Schwarz-Young on $\Delta t (g^{n+1},T^{n+1}_{h})$,
\begin{align}
\Delta t (g^{n+1},T^{n+1}_{h}) \leq \frac{\Delta t}{2\epsilon} \|g^{n+1}\|^{2}_{-1} + \frac{\Delta t \epsilon}{2} \|\nabla T^{n+1}_{h}\|^{2}. \label{stability:thick:estg}
\end{align}
Consider $-\Delta t b^{\ast}({u'}^{n}_{h},T^{n}_{h},T^{n+1}_{h})$.  Add and subtract $-\Delta t b^{\ast}({u'}^{n}_{h},T^{n}_{h},T^{n}_{h})$, use skew-symmetry, Lemma \ref{l1}, the inverse inequality, and the Cauchy-Schwarz-Young inequality.  Then,
\begin{align}
|-\Delta t b^{\ast}({u'}^{n}_{h},T^{n}_{h},T^{n+1}_{h})| &= |-\Delta t b^{\ast}({u'}^{n}_{h},T^{n}_{h},T^{n+1}_{h} - T^{n}_{h})| \label{stability:thick:estbstar}
\\ &\leq \Delta t C_{6} \|\nabla {u'}^{n}_{h}\| \|\nabla T^{n}_{h}\| \sqrt{\|T^{n+1}_{h} - T^{n}_{h}\| \| \nabla (T^{n+1}_{h} - T^{n}_{h})\|} \notag
\\ &\leq  \frac{\Delta t C_{6} C^{1/2}_{inv,2}}{h^{1/2}} \|\nabla {u'}^{n}_{h}\| \|\nabla T^{n}_{h}\| \|T^{n+1}_{h} - T^{n}_{h}\| \notag
\\ &\leq \frac{C_{6}^{2} C_{inv,2} \Delta t^{2}}{h} \|\nabla {u'}^{n}_{h}\|^{2} \|\nabla T^{n}_{h}\|^{2} + \frac{1}{4} \|T^{n+1}_{h} - T^{n}_{h}\|^{2}. \notag 
\end{align}
\noindent Using (\ref{stability:thick:estg}) and (\ref{stability:thick:estbstar}) in (\ref{stability:thick}) leads to
\begin{multline*}
\frac{1}{2} \Big\{\|T^{n+1}_{h}\|^{2} - \|T^{n}_{h}\|^{2} + \|T^{n+1}_{h} - T^{n}_{h}\|^{2}\Big\} + \kappa \Delta t \|\nabla T^{n+1}_{h}\|^{2} \leq \frac{\Delta t}{2\epsilon} \|g^{n+1}\|^{2}_{-1} 
\\ + \frac{\Delta t \epsilon}{2} \|\nabla T^{n+1}_{h}\|^{2} + \frac{C_{6}^{2} C_{inv,2} \Delta t^{2}}{h} \|\nabla {u'}^{n}_{h}\|^{2} \|\nabla T^{n}_{h}\|^{2} + \frac{1}{4} \|T^{n+1}_{h} - T^{n}_{h}\|^{2}.
\end{multline*}
\noindent Let $\epsilon = \kappa$, add and subtract $\frac{\kappa \Delta t}{2}\|\nabla T^{n}_{h}\|^{2}$ to the l.h.s.  Regrouping terms leads to
\begin{multline*}
\frac{1}{2} \Big\{\|T^{n+1}_{h}\|^{2} - \|T^{n}_{h}\|^{2}\Big\} +\frac{1}{4}\|T^{n+1}_{h} - T^{n}_{h}\|^{2} + \frac{\kappa \Delta t}{2} \Big\{ \|\nabla T^{n+1}_{h}\|^{2} - \|\nabla T^{n}_{h}\|^{2} \Big\}
\\ + \frac{\kappa \Delta t}{2} \|\nabla T^{n}_{h} \|^{2} \Big[1 - \frac{2 C_{6}^{2} C_{inv,2} \Delta t}{\kappa h} 
\|\nabla {u'}^{n}_{h}\|^{2} \Big] \leq \frac{\Delta t}{2\kappa} \|g^{n+1}\|^{2}_{-1}.
\end{multline*}
\noindent By hypothesis, $\frac{2 C_{6}^{2} C_{inv,2} \Delta t}{\kappa h} \|\nabla {u'}^{n}_{h}\|^{2} \leq 1$.  Thus,
\begin{equation*}
\frac{1}{2} \Big\{\|T^{n+1}_{h}\|^{2} - \|T^{n}_{h}\|^{2}\Big\} +\frac{1}{4}\|T^{n+1}_{h} - T^{n}_{h}\|^{2} + \frac{\kappa \Delta t}{2} \Big\{ \|\nabla T^{n+1}_{h}\|^{2} - \|\nabla T^{n}_{h}\|^{2} \Big\} \leq \frac{\Delta t}{2\kappa} \|g^{n+1}\|^{2}_{-1}.
\end{equation*}
\noindent Sum from $n = 0$ to $n = N-1$ and put all data on the right hand side.  This yields
\begin{equation} \label{stability:thick:temp}
\frac{1}{2}\|T^{N}_{h}\|^{2} +\frac{1}{4}\sum_{n = 0}^{N-1}\|T^{n+1}_{h} - T^{n}_{h}\|^{2} + \frac{\kappa \Delta t}{2} \|\nabla T^{N}_{h}\|^{2} \leq \frac{\Delta t}{2\kappa} \sum_{n = 0}^{N-1} \|g^{n+1}\|^{2}_{-1} + \frac{1}{2}\|T^{0}_{h}\|^{2} + \frac{\kappa \Delta t}{2} \|\nabla T^{0}_{h}\|^{2}.
\end{equation}
\noindent Therefore, the l.h.s. is bounded by data on the r.h.s. The temperature approximation is stable.

We follow an almost identical form of attack for the velocity as we did for the temperature.  Let $v_{h} = u^{n+1}_{h} \in V_{h}$ in (\ref{scheme:one:velocity}) and use the polarization identity.  Multiply by $\Delta t$ on both sides and rearrange terms.  Then,
\begin{multline} \label{stability:thicku}
\frac{1}{2} \Big\{\|u^{n+1}_{h}\|^{2} - \|u^{n}_{h}\|^{2} + \|u^{n+1}_{h} - u^{n}_{h}\|^{2}\Big\} + \Delta t Pr \|\nabla u^{n+1}_{h}\|^{2} 
\\ = -\Delta t b({u'}^{n}_{h},u^{n}_{h},u^{n+1}_{h}) + \Delta t PrRa(\gamma T^{n+1}_{h}, u^{n+1}_{h}) + \Delta t (f^{n+1},u^{n+1}_{h}).
\end{multline}
%\noindent where $b(<u_{h}>^{n},u^{n+1}_{h},u^{n+1}_{h}) = 0$ by skew-symmetry, $(p^{n+1}_{h}, \nabla \cdot u%^{n+1}_{h}) = 0$ since $u^{n+1}_{h} \in V_{h}$.  
Use the Cauchy-Schwarz-Young inequality on $\Delta t PrRa(\gamma T^{n+1}_{h}, u^{n+1}_{h})$ and $\Delta t (f^{n+1},u^{n+1}_{h})$ and note that $\lvert \gamma \rvert = 1$,
\begin{align}
\Delta t PrRa(\gamma T^{n+1}_{h}, u^{n+1}_{h}) &\leq \frac{\Delta t Pr^{2}Ra^{2}C_{PF,1}^2}{2 \epsilon} \|T^{n+1}_{h} \|^{2} + \frac{\Delta t \epsilon}{2} \| \nabla u^{n+1}_{h} \|^{2}, \label{stability:thick:estT}
\\ \Delta t (f^{n+1},u^{n+1}_{h}) &\leq \frac{\Delta t}{2 \epsilon} \|f^{n+1} \|^{2}_{-1} + \frac{\Delta t \epsilon}{2} \| \nabla u^{n+1}_{h} \|^{2}. \label{stability:thick:estf}
\end{align}
Using skew-symmetry, Lemma \ref{l1}, the inverse inequality, and the Cauchy-Schwarz-Young inequality on $\Delta t b({u'}^{n}_{h},u^{n}_{h},u^{n+1}_{h})$ leads to
\begin{align}
|-\Delta t b({u'}^{n}_{h},u^{n}_{h},u^{n+1}_{h})| \leq \frac{C_{5}^{2} C_{inv,1} \Delta t^{2}}{h} \|\nabla {u'}^{n}_{h}\|^{2} \|\nabla u^{n}_{h}\|^{2} + \frac{1}{4} \|u^{n+1}_{h} - u^{n}_{h}\|^{2}. \label{stability:thick:estb}
\end{align}
Using (\ref{stability:thick:estT}), (\ref{stability:thick:estf}), and (\ref{stability:thick:estb}) in (\ref{stability:thicku}) leads to
\begin{multline*}
\frac{1}{2} \Big\{\|u^{n+1}_{h}\|^{2} - \|u^{n}_{h}\|^{2} + \|u^{n+1}_{h} - u^{n}_{h}\|^{2}\Big\} + Pr \Delta t \|\nabla u^{n+1}_{h}\|^{2} + \leq \frac{\Delta t Pr^{2}Ra^{2}C_{PF,1}^2}{2 \epsilon} \|T^{n+1}_{h} \|^{2} + \frac{\Delta t}{2 \epsilon} \|f^{n+1} \|^{2}_{-1}
\\ + \Delta t \epsilon \|\nabla u^{n+1}_{h} \|^{2} + \frac{C_{5}^{2} C_{inv,1} \Delta t^{2}}{h} \|\nabla {u'}^{n}_{h}\|^{2} \|\nabla u^{n}_{h}\|^{2} + \frac{1}{4} \|u^{n+1}_{h} - u^{n}_{h}\|^{2}.
\end{multline*}
Let $\epsilon = Pr/2$, add and subtract $\frac{Pr \Delta t}{2}\|\nabla u^{n}_{h}\|^{2}$ to the l.h.s., and regroup terms.  Then,
\begin{multline*}
\frac{1}{2} \Big\{\|u^{n+1}_{h}\|^{2} - \|u^{n}_{h}\|^{2}\Big\} +\frac{1}{4}\|u^{n+1}_{h} - u^{n}_{h}\|^{2} + \frac{Pr \Delta t}{2} \Big\{ \|\nabla u^{n+1}_{h}\|^{2} - \|\nabla u^{n}_{h}\|^{2} \Big\}
\\ + \frac{Pr \Delta t}{2} \|\nabla u^{n}_{h} \|^{2} \Big[1 - \frac{2 C_{5}^{2} C_{inv,1}  \Delta t}{Pr h} 
\|\nabla {u'}^{n}_{h}\|^{2} \Big] \leq \Delta t PrRa^{2} C_{PF,1}^2 \|T^{n+1}_{h} \|^{2} + \frac{\Delta t}{Pr} \|f^{n+1} \|^{2}_{-1}.
\end{multline*}
\noindent By hypothesis, $\frac{2 C_{5}^{2} C_{inv,1} \Delta t}{Pr h} \|\nabla {u'}^{n}_{h}\|^{2} \leq 1$.  Thus,
\begin{multline*}
\frac{1}{2} \Big\{\|u^{n+1}_{h}\|^{2} - \|u^{n}_{h}\|^{2}\Big\} +\frac{1}{4}\|u^{n+1}_{h} - u^{n}_{h}\|^{2} + \frac{Pr \Delta t}{2} \Big\{ \|\nabla u^{n+1}_{h}\|^{2} - \|\nabla u^{n}_{h}\|^{2} \Big\} 
\\ \leq \Delta t PrRa^{2}C_{PF,1}^2 \|T^{n+1}_{h} \|^{2} + \frac{\Delta t}{Pr} \|f^{n+1} \|^{2}_{-1}.
\end{multline*}
\noindent Summing from $n = 0$ to $n = N-1$ and putting all data on r.h.s. yields
\begin{multline} \label{stability:thick:vel}
\frac{1}{2}\|u^{N}_{h}\|^{2} +\frac{1}{4}\sum_{n = 0}^{N-1}\|u^{n+1}_{h} - u^{n}_{h}\|^{2} + \frac{Pr \Delta t}{2} \|\nabla u^{N}_{h}\|^{2} \leq \Delta t PrRa^{2}C_{PF,1}^2 \sum_{n=0}^{N-1} \|T^{n+1}_{h} \|^{2} + \frac{\Delta t}{Pr} \sum_{n=0}^{N-1} \|f^{n+1} \|^{2}_{-1} 
\\ + \frac{1}{2}\|u^{0}_{h}\|^{2} + \frac{Pr \Delta t}{2} \|\nabla u^{0}_{h}\|^{2}.
\end{multline}
\noindent Together with the stability of the temperature approximation, the l.h.s. is bounded above by data; that is, the velocity approximation is stable.  Adding (\ref{stability:thick:temp}) and (\ref{stability:thick:vel}) and mutliplying by 2 yields the result.  We now prove stability of the pressure approximation.  We first form an estimate for the discrete time derivative term.  Consider (\ref{scheme:one:velocity}), isolate $(\frac{u^{n+1}_{h} - u^{n}_{h}}{\Delta t},v_{h})$, let $0 \neq v_{h} \in V_{h}$, and multiply by $\Delta t$.  Then,
\begin{multline}\label{vv}
	({u^{n+1}_{h} - u^{n}_{h}},v_{h}) = -\Delta t b(<u_{h}>^{n},u^{n+1}_{h},v_{h}) - \Delta t b({u'}^{n}_{h},u^{n}_{h},v_{h}) 
	\\ - \Delta t Pr(\nabla u^{n+1}_{h},\nabla v_{h}) + \Delta t PrRa(\gamma T^{n+1}_{h},v_{h}) + \Delta t (f^{n+1},v_{h}).
\end{multline}
Applying Lemma \ref{l1} to the skew-symmetric trilinear terms and the Cauchy-Schwarz and Poincar\'{e}-Friedrichs inequalities to the remaining terms yields
\begin{align}
|-\Delta t b(<u_{h}>^{n},u^{n+1}_{h},v_{h})| &\leq C_{1} \Delta t \| \nabla <u_{h}>^{n} \| \| \nabla u^{n+1}_{h} \| \| \nabla v_{h}\|,\label{stability:thick:estmean}\\
|-\Delta t b({u'}^{n}_{h},u^{n}_{h},v_{h})| &\leq C_{1} \Delta t \| \nabla {u'}^{n}_{h} \| \| \nabla  u^{n}_{h} \| \| \nabla v_{h}\|,\label{stability:thick:estfluc}\\
|-\Delta t Pr(\nabla u^{n+1}_{h},\nabla v_{h})| &\leq Pr \Delta t \|\nabla u^{n+1}_{h} \| \|\nabla v_{h} \|, \label{stability:thick:estvisc}\\
|\Delta t PrRa(\gamma T^{n+1}_{h},v_{h})| &\leq PrRa \Delta t \|T^{n+1}_{h} \| \|v_{h} \| \leq PrRaC_{PF,1} \Delta t \|T^{n+1}_{h} \| \|\nabla v_{h} \|, \label{stability:thick:estcoupling}\\
|\Delta t(f^{n+1},v_{h})| &\leq \Delta t \|f^{n+1} \|_{-1} \|\nabla v_{h} \| \label{stability:thick:estf2}.
\end{align}
Apply the above estimates in (\ref{vv}), divide by the common factor $\|\nabla v_{h} \|$ on both sides, and take the supremum over all $0 \neq v_{h} \in V_{h}$.  Then,
\begin{multline}
\|u^{n+1}_{h} - u^{n}_{h}\|_{V^{\ast}_{h}} \leq C_{1} \Delta t \| \nabla <u_{h}>^{n} \| \| \nabla u^{n+1}_{h} \| + C_{1} \Delta t \|\nabla {u'}^{n}_{h} \| \| \nabla  u^{n}_{h} \| 
\\ + Pr \Delta t \|\nabla u^{n+1}_{h} \| +  PrRaC_{PF,1} \Delta t \|T^{n+1}_{h} \| + \Delta t \|f^{n+1} \|_{-1}.
\end{multline}
\noindent Reconsider equation (\ref{scheme:one:velocity}).  Multiply by $\Delta t$ and isolate the pressure term,
\begin{multline}
\Delta t (p^{n+1}_{h}, \nabla \cdot v_{h}) = (u^{n+1}_{h} - u^{n}_{h},v_{h}) + \Delta t b(<u_{h}>^{n},u^{n+1}_{h},v_{h}) + \Delta t b({u'}^{n}_{h},u^{n}_{h},v_{h}) 
\\ + Pr \Delta t(\nabla u^{n+1}_{h},\nabla v_{h}) - PrRa \Delta t(\gamma T^{n+1}_{h},v_{h}) - \Delta t(f^{n+1},v_{h}).
\end{multline}
\noindent Apply (\ref{stability:thick:estmean}) - (\ref{stability:thick:estf2}) on the r.h.s terms.  Then,
\begin{multline}
\Delta t (p^{n+1}_{h}, \nabla \cdot v_{h}) \leq (u^{n+1}_{h} - u^{n}_{h},v_{h}) + \Big(C_{1} \Delta t \| \nabla <u_{h}>^{n} \| \| \nabla u^{n+1}_{h} \| + C_{1} \Delta t \|\nabla {u'}^{n}_{h} \| \| \nabla  u^{n}_{h} \| 
\\ + Pr \Delta t \|\nabla u^{n+1}_{h} \| +  PrRaC_{PF,1} \Delta t \|T^{n+1}_{h} \| + \Delta t \|f^{n+1} \|_{-1}\Big)\|\nabla v_{h} \|.
\end{multline}
\noindent Divide by $\|\nabla v_{h} \|$ and note that $\frac{(u^{n+1}_{h} - u^{n}_{h},v_{h})}{\| \nabla v_{h} \|} \leq \|u^{n+1}_{h} - u^{n}_{h}\|_{V^{\ast}_{h}}$.  Take the supremum over all $0 \neq v_{h} \in X_{h}$,
\begin{multline}
\Delta t \sup_{0 \neq v_{h} \in X_{h}}\frac{(p^{n+1}_{h}, \nabla \cdot v_{h})}{\|\nabla v_{h} \|} \leq 2 \Big(C_{1} \Delta t \| \nabla <u_{h}>^{n} \| \| \nabla u^{n+1}_{h} \| + C_{1} \Delta t \|\nabla {u'}^{n}_{h} \| \| \nabla  u^{n}_{h} \| 
\\ + Pr \Delta t \|\nabla u^{n+1}_{h} \| +  PrRaC_{PF,1} \Delta t \|T^{n+1}_{h} \| + \Delta t \|f^{n+1} \|_{-1}\Big).
\end{multline}
\noindent Use the discrete inf-sup condition (\ref{infsup}),
\begin{multline}
\beta \Delta t \| p^{n+1}_{h}\| \leq 2 \Big( C_{1} \Delta t \| \nabla <u_{h}>^{n} \| \| \nabla u^{n+1}_{h} \| + C_{1} \Delta t \|\nabla {u'}^{n}_{h} \| \| \nabla  u^{n}_{h} \| 
\\ + Pr \Delta t \|\nabla u^{n+1}_{h} \| +  PrRaC_{PF,1} \Delta t \|T^{n+1}_{h} \| + \Delta t \|f^{n+1} \|_{-1}\Big).
\end{multline}
\noindent Summing from $n = 0$ to $n = N-1$ yields stability of the pressure approximation, built on the stability of the temperature and velocity approximations.
\end{proof}

\begin{theorem} \label{t2}
Consider the \textbf{Thin wall problem} (\ref{scheme:two:velocity}) - (\ref{scheme:two:temperature}).  Suppose $f \in L^{\infty}(0,t^{\ast};H^{-1}(\Omega)^{d})$ and $g \in L^{\infty}(0,t^{\ast};H^{-1}(\Omega))$.  If (\ref{scheme:two:velocity}) - (\ref{scheme:two:temperature}) satisfy condition (\ref{c1}), then
\begin{multline*}
\|T^{N}_{h}\|^{2} + \|u^{N}_{h}\|^{2}  + \frac{1}{2}\sum_{n = 0}^{N-1}\big(\|T^{n+1}_{h} - T^{n}_{h}\|^{2} + \|u^{n+1}_{h} - u^{n}_{h}\|^{2}\big) + \kappa \Delta t \|\nabla T^{N}_{h}\|^{2} + Pr \Delta t \|\nabla u^{N}_{h}\|^{2} \\
\leq exp(2 C t^{\ast}) \Big\{  \Delta t \sum_{n = 0}^{N-1} ( \frac{1}{Pr} \| f^{n+1} \|_{-1}^{2} + \frac{1}{\kappa} \| g^{n+1} \|_{-1}^{2})  + \|u^{0}_{h} \|^{2} + \| T^{0}_{h} \|^{2} 
\\ + Pr \Delta t \|\nabla u^{0}_{h}\|^{2} + \kappa \Delta t \|\nabla T^{0}_{h}\|^{2}\Big\}.
\end{multline*}
\noindent Further,
\begin{multline*}
\beta \Delta t \sum^{N-1}_{n=0} \| p^{n+1}_{h}\| \leq 2 \sum^{N-1}_{n=0} \Big( C_{1} \Delta t \| \nabla <u_{h}>^{n} \| \| \nabla u^{n+1}_{h} \| + C_{1} \Delta t \|\nabla {u'}^{n}_{h} \| \| \nabla  u^{n}_{h} \| 
\\ + Pr \Delta t \|\nabla u^{n+1}_{h} \| +  PrRaC_{PF,1} \Delta t \|T^{n}_{h} \| + \Delta t \|f^{n+1} \|_{-1}\Big).
\end{multline*}
\end{theorem}
\begin{proof}
Add equations (\ref{scheme:two:velocity}) and (\ref{scheme:two:temperature}), let $S_{h} = T^{n+1}_{h} \in W_{h}$ and $v_{h} = u^{n+1}_{h} \in V_{h}$ and use the polarization identity.  Then, 
\begin{align}
&\frac{1}{2 \Delta t} \Big\{\|T^{n+1}_{h}\|^{2} - \|T^{n}_{h}\|^{2} + \|T^{n+1}_{h} - T^{n}_{h}\|^{2}\Big\} + \frac{1}{2 \Delta t} \Big\{\|u^{n+1}_{h}\|^{2} - \|u^{n}_{h}\|^{2} + \|u^{n+1}_{h} - u^{n}_{h}\|^{2}\Big\}
\\ &+ \kappa \|\nabla T^{n+1}_{h}\|^{2} + Pr \|\nabla u^{n+1}_{h}\|^{2} + b({u'}^{n}_{h},u^{n}_{h},u^{n+1}_{h}) + b^{\ast}({u'}^{n}_{h},T^{n}_{h},T^{n+1}_{h}) = PrRa(\gamma T^{n}_{h}, u^{n+1}_{h}) \notag
\\ &+ (u^{n}_{1h},T^{n+1}_{h}) + (f^{n+1},u^{n+1}_{h}) + (g^{n+1},T^{n+1}_{h}). \notag
\end{align}
Apply similar techniques and estimates as in the proof of Theorem \ref{t1},
\begin{align}
&\frac{1}{2} \Big\{\|T^{n+1}_{h}\|^{2} - \|T^{n}_{h}\|^{2} + \frac{1}{2} \|T^{n+1}_{h} - T^{n}_{h}\|^{2}\Big\} + \frac{1}{2} \Big\{\|u^{n+1}_{h}\|^{2} - \|u^{n}_{h}\|^{2} + \frac{1}{2} \|u^{n+1}_{h} - u^{n}_{h}\|^{2}\Big\}
\\ &+ \frac{\kappa \Delta t}{2} \big\{ \|\nabla T^{n+1}_{h}\|^{2} - \|\nabla T^{n}_{h}\|^{2} \big\} + \frac{Pr \Delta t}{2} \big\{\|\nabla u^{n+1}_{h}\|^{2} - \|\nabla u^{n}_{h}\|^{2} \big\} \notag
\\ &+ \frac{\kappa \Delta t}{2} \| \nabla T^{n}_{h} \|^{2} \big\{ 1- \frac{2 \Delta t C_{6}^2 C_{inv,2}}{\kappa h} \| \nabla {u'}^{n}_{h} \|^{2} \big\} + \frac{Pr \Delta t}{2} \| \nabla u^{n}_{h} \|^{2} \big\{ 1- \frac{2 \Delta t C_{5}^2 C_{inv,1}}{Pr h} \| \nabla {u'}^{n}_{h} \|^{2} \big\} \notag
\\ &\leq {\Delta tPr Ra^{2} C_{PF,1}^{2}} \| T^{n}_{h} \|^{2} + \frac{\Delta t C_{PF,2}^{2}}{\kappa} \| u^{n}_{h} \|^{2}+ \frac{\Delta t}{Pr} \| f^{n+1} \|_{-1}^{2} + \frac{\Delta t}{\kappa} \| g^{n+1} \|_{-1}^{2}. \notag
\end{align}
\noindent Using the timestep condition, multiplying by 2, taking a maximum over constants in the first two terms on the r.h.s. and summing from $n = 0$ to $n = N-1$ leads to,
\begin{multline}
\|T^{N}_{h}\|^{2} + \|u^{N}_{h}\|^{2} + \frac{1}{2}\sum_{n = 0}^{N-1}\big(\|T^{n+1}_{h} - T^{n}_{h}\|^{2} + \|u^{n+1}_{h} - u^{n}_{h}\|^{2}\big) + \kappa \Delta t \|\nabla T^{N}_{h}\|^{2} + Pr \Delta t \|\nabla u^{N}_{h}\|^{2}\\
\leq C \Delta t \sum_{n = 0}^{N-1} \big\{ \| T^{n}_{h} \|^{2} +  \| u^{n}_{h} \|^{2} \big\} + 2 \Delta t \sum_{n = 0}^{N-1} \big\{ \frac{1}{Pr} \| f^{n+1} \|_{-1}^{2} + \frac{1}{\kappa} \| g^{n+1} \|_{-1}^{2} \big\} + \|u^{0}_{h} \|^{2} + \| T^{0}_{h} \|^{2} \\
+ Pr \Delta t \|\nabla u^{0}_{h}\|^{2} + \kappa \Delta t \|\nabla T^{0}_{h}\|^{2}.
\end{multline}
Lastly, apply Lemma \ref{l4}.  Then,
\begin{multline}
\|T^{N}_{h}\|^{2} + \|u^{N}_{h}\|^{2} + \frac{1}{2}\sum_{n = 0}^{N-1}\big(\|T^{n+1}_{h} - T^{n}_{h}\|^{2} + \|u^{n+1}_{h} - u^{n}_{h}\|^{2}\big) + \kappa \Delta t \|\nabla T^{N}_{h}\|^{2} + Pr \Delta t \|\nabla u^{N}_{h}\|^{2}\\
\leq exp(C t^{\ast}) \Big\{ 2 \Delta t \sum_{n = 0}^{N-1} ( \frac{1}{Pr} \| f^{n+1} \|_{-1}^{2} + \frac{1}{\kappa} \| g^{n+1} \|_{-1}^{2})  + \|u^{0}_{h} \|^{2} + \| T^{0}_{h} \|^{2} \\
+ Pr \Delta t \|\nabla u^{0}_{h}\|^{2} + \kappa \Delta t \|\nabla T^{0}_{h}\|^{2}\Big\}.
\end{multline}
Thus, numerical approximations of velocity and temperature are stable.  Stability of the pressure approximation follows by similar arguments as in Theorem \ref{t1}.
\end{proof}
\textbf{Remark:}  Theorem \ref{t1} implies long-time stability of the approximate solutions.  Application of Lemma \ref{l4} in Theorem \ref{t2} leads to the loss of long-time stability due to the exponential growth factor, in $t^{\ast}$.
%%%%%%%%%%%%%%%%%%%%%%%%%%%%%%%%%%%%
%%%%%%%%%%%%%%%%%%%%%%%%%%%%%%%%%%%%
\subsection{Error Analysis}
Denote $u^{n}$, $p^{n}$, and $T^{n}$ as the true solutions at time $t^{n} = n\Delta t$.  Assume the solutions satisfy the following regularity assumptions:
\begin{align} 
u &\in L^{\infty}(0,t^{\ast};X \cap H^{k+1}(\Omega)), \;  T \in L^{\infty}(0,t^{\ast};W \cap H^{k+1}(\Omega)), \notag \\
u_{t}, T_{t} &\in L^{\infty}(0,t^{\ast};H^{k+1}(\Omega)), \; u_{tt}, T_{tt} \in L^{\infty}(0,t^{\ast};L^{2}(\Omega)), \label{error:regularity} \\
p &\in L^{\infty}(0,t^{\ast};Q \cap H^{m}(\Omega)). \notag
\end{align}
\noindent The errors are denoted
\begin{align*}
e^{n}_{u} &= u^{n} - u^{n}_{h}, \; e^{n}_{T} = T^{n} - T^{n}_{h}, \; e^{n}_{p} = p^{n} - p^{n}_{h}.
\end{align*}
\begin{definition} (Consistency error).  The consistency errors are defined as
\begin{align*}
\tau_{u}(u^{n};v_{h}) = \big(\frac{u^{n}-u^{n-1}}{\Delta t} - u^{n}_{t}, v_{h}\big), \; \tau_{T}(T^{n};S_{h}) = \big(\frac{T^{n}-T^{n-1}}{\Delta t} - T^{n}_{t}, S_{h}\big).
\end{align*}
\end{definition}
\begin{lemma}\label{consistency}
Provided  $u$ and $T$ satisfy the regularity assumptions (\ref{error:regularity}), then $\forall r > 0$
\begin{align*}
\lvert \tau_{u}(u^{n};v_{h}) \rvert &\leq \frac{C^{2}_{PF,1} C_{r} \Delta t^{2}}{\epsilon}\| u_{tt}\|^{2}_{L^{\infty}(t^{n-1},t^{n};L^{2}(\Omega))} + \frac{\epsilon}{r} \| \nabla v_{h} \|^{2},
\\ \lvert \tau_{T}(T^{n};S_{h}) \rvert &\leq \frac{C^{2}_{PF,2} C_{r} \Delta t^{2}}{\epsilon}\| T_{tt}\|^{2}_{L^{\infty}(t^{n-1},t^{n};L^{2}(\Omega))} + \frac{\epsilon}{r} \| \nabla S_{h} \|^{2}.
\end{align*}
\end{lemma}
\begin{proof}
These follow from the Cauchy-Schwarz-Young inequality, Poincar\'{e}-Friedrichs inequality, and Taylor's Theorem with integral remainder.
\end{proof}
\begin{theorem} \label{error:thick}
For (u,p,T) satisfying (1) - (5), suppose that $(u^{0}_{h},p^{0}_{h},T^{0}_{h}) \in (X_{h},Q_{h},W_{h})$ are approximations of $(u^{0},p^{0},T^{0})$ to within the accuracy of the interpolant.  Further, suppose that condition (\ref{c1}) holds. Then there exists a constant C such that

\begin{multline*}
\|e^{N}_{T}\|^{2} + \|e^{N}_{u}\|^{2} + \frac{1}{2}\sum_{n = 0}^{N-1}\big(\|e^{n+1}_{T} - e^{n}_{T}\|^{2} + \|e^{n+1}_{u} - e^{n}_{u}\|^{2}\big) + \frac{\kappa \Delta t}{2}\|\nabla e^{N}_{T}\|^{2} + \frac{Pr \Delta t}{2} \|\nabla e^{N}_{u}\|^{2}
\\ \leq  C \Big\{ \Delta t \inf_{v_{h} \in X_{h}} \Big( \vertiii{\nabla (u - v_{h}) }^{2}_{\infty,0} + \vertiii{(u - v_{h})_{t}}^{2}_{\infty,0} \Big) + \Delta t \inf_{S_{h} \in W_{h}} \Big( \vertiii{\nabla (T - S_{h})}^{2}_{\infty,0} + \vertiii{(T - S_{h})_{t}}^{2}_{\infty,0} \Big)
\\ + \Delta t \inf_{q_{h} \in Q_{h}} \vertiii{ p - q_{h} }^{2}_{\infty,0} + \Delta t^{3}  + \Delta t \|\nabla \eta^{0}\|^{2} + \Delta t \|\nabla \zeta^{0}\|^{2} + \|\eta^{0}\|^{2} 
\\ + \|\zeta^{0}\|^{2} + \| e^{0}_{T} \|^{2} + \| e^{0}_{u} \|^{2} + \Delta t \| \nabla e^{0}_{T} \|^{2} + \Delta t\| \nabla e^{0}_{u} \|^{2}\Big\}.
\end{multline*}

\end{theorem}
\begin{proof}
The true solutions satisfy for all $n = 0, 1, ... N$:
\begin{align}
(\frac{u^{n+1} - u^{n}}{\Delta t},v_{h}) + b(u^{n+1},u^{n+1},v_{h}) + Pr(\nabla u^{n+1},\nabla v_{h}) - (p^{n+1}, \nabla \cdot v_{h}) =  PrRa(\gamma T^{n+1},v_{h})
\label{error:one:truevelocity} \\ + (f^{n+1},v_{h}) + \tau_{u}(u^{n+1};v_{h}) \; \; \forall v_{h} \in X_{h} \notag,
\\ (q_{h}, \nabla \cdot u^{n+1}) = 0 \; \; \forall q_{h} \in Q_{h},
\\ (\frac{T^{n+1} - T^{n}}{\Delta t},S_{h}) + b^{\ast}(u^{n+1},T^{n+1},S_{h}) + \kappa (\nabla T^{n+1},\nabla S_{h}) = (g^{n+1},S_{h}) + \tau_{T}(T^{n+1};S_{h}) \; \; \forall S_{h} \in W_{h} \label{error:one:truetemp}.
\end{align}
Subtract (\ref{error:one:truetemp}) and (\ref{scheme:one:temperature}), then the error equation for temperature is
\begin{align}
(\frac{e^{n+1}_{T} - e^{n}_{T}}{\Delta t},S_{h}) + b^{\ast}(u^{n+1},T^{n+1},S_{h}) - b^{\ast}(u^{n}_{h}-{u'}^{n}_{h},T^{n+1}_{h},S_{h}) - b^{\ast}({u'}^{n}_{h},T^{n}_{h},S_{h})
\\ + \kappa (\nabla e^{n+1}_{T},\nabla S_{h}) = \tau_{T}(T^{n+1},S_{h}) \; \; \forall S_{h} \in W_{h}. \notag
\end{align}
Letting $e^{n}_{T} = (T^{n} - \tilde{T}^{n}) - (T^{n}_{h}- \tilde{T}^{n}) = \zeta^{n} - \psi^{n}_{h}$  and rearranging give,
\begin{multline*}
(\frac{\psi^{n+1}_{h} - \psi^{n}_{h}}{\Delta t},S_{h}) + \kappa (\nabla \psi^{n+1}_{h},\nabla S_{h}) = (\frac{\zeta^{n+1} - \zeta^{n}}{\Delta t},S_{h}) + \kappa (\nabla \zeta^{n+1},\nabla S_{h}) - \tau_{T}(T^{n+1},S_{h})
\\ + b^{\ast}(u^{n+1},T^{n+1},S_{h}) - b^{\ast}(u^{n}_{h}-{u'}^{n}_{h},T^{n+1}_{h},S_{h}) - b^{\ast}({u'}^{n}_{h},T^{n}_{h},S_{h})  \; \; \forall S_{h} \in W_{h}.
\end{multline*}
Set $S_{h} = \psi^{n+1}_{h} \in W_{h}.$  This yields
\begin{multline}
\frac{1}{2 \Delta t} \Big\{\|\psi^{n+1}_{h}\|^{2} - \|\psi^{n}_{h}\|^{2} + \|\psi^{n+1}_{h} - \psi^{n}_{h}\|^{2}\Big\} + \kappa \|\nabla \psi^{n+1}_{h}\|^{2} = \frac{1}{\Delta t}(\zeta^{n+1}-\zeta^{n},\psi^{n+1}_{h}) + \kappa (\nabla \zeta^{n+1},\nabla \psi^{n+1}_{h}) 
\\ - \tau_{T}(T^{n+1},\psi^{n+1}_{h}) + b^{\ast}(u^{n+1},T^{n+1},\psi^{n+1}_{h}) - b^{\ast}(u^{n}_{h}-{u'}^{n}_{h},T^{n+1}_{h},\psi^{n+1}_{h}) - b^{\ast}({u'}^{n}_{h},T^{n}_{h},\psi^{n+1}_{h}).
\end{multline}
Add and subtract $b^{\ast}(u^{n+1},T^{n+1}_{h},\psi^{n+1}_{h})$, $b^{\ast}(u^{n},T^{n+1}_{h},\psi^{n+1}_{h})$, and $b^{\ast}({u'}^{n}_{h},T^{n+1}-T^{n},\psi^{n+1}_{h})$.  Then,
\begin{align}\label{fet1}
&\frac{1}{2 \Delta t} \Big\{\|\psi^{n+1}_{h}\|^{2} - \|\psi^{n}_{h}\|^{2} + \|\psi^{n+1}_{h} - \psi^{n}_{h}\|^{2}\Big\} + \kappa \|\nabla \psi^{n+1}_{h}\|^{2} = \frac{1}{\Delta t}(\zeta^{n+1}-\zeta^{n},\psi^{n+1}_{h}) + \kappa (\nabla \zeta^{n+1},\nabla \psi^{n+1}_{h})
\\ &+ b^{\ast}(u^{n+1},\zeta^{n+1},\psi^{n+1}_{h}) + b^{\ast}(u^{n+1}-u^{n},T^{n+1}_{h},\psi^{n+1}_{h}) + b^{\ast}(\eta^{n},T^{n+1}_{h},\psi^{n+1}_{h}) \notag
\\ &- b^{\ast}(\phi^{n}_{h},T^{n+1}_{h},\psi^{n+1}_{h}) + b^{\ast}({u'}^{n}_{h},\zeta^{n+1},\psi^{n+1}_{h}) - b^{\ast}({u'}^{n}_{h},\zeta^{n},\psi^{n+1}_{h}) + b^{\ast}({u'}^{n}_{h},\psi^{n}_{h},\psi^{n+1}_{h}) \notag
\\ &+ b^{\ast}({u'}^{n}_{h},T^{n+1}-T^{n},\psi^{n+1}_{h}) - \tau_{T}(T^{n+1},\psi^{n+1}_{h}). \notag
\end{align}
Follow analogously for the velocity error equation.  Subtract (\ref{error:one:truevelocity}) and (\ref{scheme:one:velocity}), split the error into $e^{n}_{u} = (u^{n} - \tilde{u}^{n}) - (u^{n}_{h}- \tilde{u}^{n}) = \eta^{n} - \phi^{n}_{h}$, let $v_{h} = \phi^{n+1}_{h} \in V_{h}$, add and subtract $b(u^{n+1},u^{n+1}_{h},\phi^{n+1}_{h})$, $b(u^{n},u^{n+1}_{h},\phi^{n+1}_{h})$, and $b({u'}^{n}_{h},u^{n+1}-u^{n},\phi^{n+1}_{h})$.  Then,
\begin{align}\label{feu1}
&\frac{1}{2 \Delta t} \Big\{\|\phi^{n+1}_{h}\|^{2} - \|\phi^{n}_{h}\|^{2} + \|\phi^{n+1}_{h} - \phi^{n}_{h}\|^{2}\Big\} + Pr \|\nabla \phi^{n+1}_{h}\|^{2} = \frac{1}{\Delta t}(\eta^{n+1}-\eta^{n},\phi^{n+1}_{h})
\\ &+ Pr (\nabla \eta^{n+1},\nabla \phi^{n+1}_{h}) - (p^{n+1} - q^{n+1}_{h},\nabla \cdot \phi^{n+1}_{h}) + PrRa(\gamma \zeta^{n+1},\phi^{n+1}_{h}) \notag
\\ &- PrRa(\gamma \psi^{n+1}_{h},\phi^{n+1}_{h})+ b(u^{n+1},\eta^{n+1},\phi^{n+1}_{h}) + b(u^{n+1}-u^{n},u^{n+1}_{h},\phi^{n+1}_{h}) + b(\eta^{n},u^{n+1}_{h},\phi^{n+1}_{h}) \notag
\\ &- b(\phi^{n}_{h},u^{n+1}_{h},\phi^{n+1}_{h}) + b({u'}^{n}_{h},\eta^{n+1},\phi^{n+1}_{h}) - b({u'}^{n}_{h},\eta^{n},\phi^{n+1}_{h}) + b({u'}^{n}_{h},\phi^{n}_{h},\phi^{n+1}_{h}) \notag
\\ &+ b({u'}^{n}_{h},u^{n+1}-u^{n},\phi^{n+1}_{h}) - \tau_{u}(u^{n+1},\phi^{n+1}_{h}). \notag
\end{align}
Our goal now is to estimate all terms on the r.h.s. in such a way that we may hide the terms involving unknown pieces $\psi^{k}_{h}$ into the l.h.s.  The following estimates are formed using Lemma \ref{l1} in conjunction with the Cauchy-Schwarz-Young inequality,
\begin{align}
\lvert b^{\ast}(u^{n+1},\zeta^{n+1},\psi^{n+1}_{h}) \rvert &\leq C_{3} \| \nabla u^{n+1} \| \| \nabla \zeta^{n+1} \| \| \nabla \psi^{n+1}_{h} \| \leq \frac{C_{r} C_{3}^{2}}{\epsilon_3} \| \nabla u^{n+1} \|^{2} \| \nabla \zeta^{n+1} \|^{2} + \frac{\epsilon_3}{r} \| \nabla \psi^{n+1}_{h} \|^{2},
\\ \lvert b^{\ast}(\eta^{n},T^{n+1}_{h},\psi^{n+1}_{h}) \rvert &\leq C_{3} \| \nabla \eta^{n} \| \| \nabla T^{n+1}_{h} \| \| \nabla \psi^{n+1}_{h} \| \leq \frac{C_{r} C_{3}^{2}}{\epsilon_5} \| \nabla \eta^{n} \|^{2} \| \nabla T^{n+1}_{h} \|^{2} + \frac{\epsilon_5}{r} \| \nabla \psi^{n+1}_{h} \|^{2},
\\ \lvert b^{\ast}({u'}^{n}_{h},\zeta^{n+1},\psi^{n+1}_{h}) \rvert &\leq C_{3} \| {u'}^{n}_{h} \| \| \zeta^{n+1} \| \| \psi^{n+1}_{h} \| \leq \frac{C_{r} C_{3}^{2}}{\epsilon_7} \| \nabla {u'}^{n}_{h} \|^{2} \| \nabla \zeta^{n+1} \|^{2} + \frac{\epsilon_7}{r} \| \nabla \psi^{n+1}_{h} \|^{2},
\\ \lvert -b^{\ast}({u'}^{n}_{h},\zeta^{n},\psi^{n+1}_{h}) \rvert &\leq C_{3}\| \nabla {u'}^{n}_{h} \| \| \nabla \zeta^{n} \| \| \nabla \psi^{n+1}_{h} \| \leq \frac{C_{r} C_{3}^{2}}{\epsilon_8} \| \nabla {u'}^{n}_{h} \|^{2} \| \nabla \zeta^{n} \|^{2} + \frac{\epsilon_8}{r} \| \nabla \psi^{n+1}_{h} \|^{2}.
\end{align}
Applying Lemma \ref{l1}, the Cauchy-Schwarz-Young inequality, and Taylor's theorem yields,
\begin{align}
\lvert b^{\ast}(u^{n+1}-u^{n},T^{n+1}_{h},\psi^{n+1}_{h}) \rvert &\leq C_{3} \| \nabla (u^{n+1}-u^{n}) \| \| \nabla T^{n+1}_{h} \| \| \nabla \psi^{n+1}_{h} \|
\\ &\leq \frac{C_{r} C_{3}^{2}}{\epsilon_4} \| \nabla (u^{n+1}-u^{n}) \|^{2} \| \nabla T^{n+1}_{h} \|^{2} + \frac{\epsilon_4}{r} \| \nabla \psi^{n+1}_{h} \|^{2} \notag
\\ &\leq \frac{C_{r} C_{3}^{2} \Delta t^{2}}{\epsilon_4} \| \nabla T^{n+1}_{h} \|^{2} \| \nabla u_{t} \|^{2}_{L^{\infty}(t^n,t^{n+1};L^{2}(\Omega))} + \frac{\epsilon_4}{r} \| \nabla \psi^{n+1}_{h} \|^{2}, \notag
\\ \lvert b^{\ast}(u^{n}_{h},T^{n+1}-T^{n},\psi^{n+1}_{h}) \rvert &\leq C_{3} \| \nabla u^{n}_{h} \| \| \nabla (T^{n+1} - T^{n}) \| \| \nabla \psi^{n+1}_{h} \|
\\ &\leq \frac{C_{r} C_{3}^{2}\Delta t^{2}}{\epsilon_{10}} \| \nabla u^{n}_{h} \|^{2} \| \nabla T_{t} \|^{2}_{L^{\infty}(t^n,t^{n+1};L^{2}(\Omega))} + \frac{\epsilon_{10}}{r} \| \nabla \psi^{n+1}_{h}\|^{2}. \notag
\end{align}
Apply Lemma \ref{l1} and the Cauchy-Schwarz-Young inequality twice.  This yields
\begin{align}
\lvert -b^{\ast}(\phi^{n}_{h},T^{n+1}_{h},\psi^{n+1}_{h}) \rvert &\leq C_{4} \sqrt{\| \phi^{n}_{h} \| \| \nabla \phi^{n}_{h} \|} \| \nabla T^{n+1}_{h} \| \| \nabla \psi^{n+1}_{h} \| \leq C_{4} C_{T}(j) \sqrt{\| \phi^{n}_{h} \| \| \nabla \phi^{n}_{h} \|} \| \nabla \psi^{n+1}_{h} \| 
\\ &\leq \frac{C_{4}C_{T}\epsilon_6}{2} \| \nabla \psi^{n+1}_{h} \|^{2} + \frac{C_{4}C_{T}\delta_6}{4 \epsilon_6} \| \nabla \phi^{n}_{h} \|^{2} + \frac{C_{4}C_{T}}{4 \epsilon_6 \delta_6} \| \phi^{n}_{h} \|^{2}. \notag
\end{align}
Use Lemma \ref{l1}, the inverse inequality, and the Cauchy-Schwarz-Young inequality yielding
\begin{align}
\lvert \Delta t b^{\ast}({u'}^{n}_{h},\psi^{n}_{h},\psi^{n+1}_{h}) \rvert &= \lvert \Delta t b^{\ast}({u'}^{n}_{h},\psi^{n}_{h},\psi^{n+1}_{h}-\psi^{n}_{h}) \rvert
\\ &\leq \Delta t C_{6} \| \nabla {u'}^{n}_{h} \| \| \nabla \psi^{n}_{h} \| \sqrt{\| \psi^{n+1}_{h}-\psi^{n}_{h} \| \| \nabla (\psi^{n+1}_{h}-\psi^{n}_{h}) \|} \notag
\\ &\leq \frac{\Delta t C_{6} C^{1/2}_{inv,2}}{h^{1/2}} \| \nabla {u'}^{n}_{h} \| \| \nabla \psi^{n}_{h} \| \| \psi^{n+1}_{h}-\psi^{n}_{h} \| \notag
\\ &\leq \frac{C_{6}^{2} C_{inv,2} \Delta t}{2 h \epsilon_9} \| \nabla {u'}^{n}_{h} \|^{2} \| \nabla \psi^{n}_{h} \|^{2} + \frac{\epsilon_9}{2} \| \psi^{n+1}_{h} - \psi^{n}_{h} \|^{2}. \notag
\end{align}
The Cauchy-Schwarz-Young inequality, Poincar\'{e}-Friedrichs inequality and Taylor's theorem yield
\begin{align}
\lvert \frac{1}{\Delta t} (\zeta^{n+1} - \zeta^{n}, \psi^{n+1}_{h}) \rvert \leq \frac{C^{2}_{PF,2} C_{r}}{\epsilon_1} \| \zeta_{t} \|^{2}_{L^{\infty}(t^n,t^{n+1};L^{2}(\Omega))} + \frac{\epsilon_1}{r} \| \nabla \psi^{n+1}_{h} \|^{2}.
\end{align}
Lastly, use the Cauchy-Schwarz-Young inequality,
\begin{align}
\lvert \kappa (\nabla \zeta^{n+1},\nabla \psi^{n+1}_{h}) \rvert \leq \frac{C_{r}\kappa^{2}}{\epsilon_2} \| \nabla \zeta^{n+1} \|^{2} + \frac{\epsilon_2}{r} \| \nabla \psi^{n+1}_{h} \|^{2}.
\end{align}
Similar estimates follow for the r.h.s. terms in (\ref{feu1}), however, we must treat an additional pressure term and error term,
\begin{align}
\lvert -(p^{n+1}-q^{n+1}_{h},\nabla \cdot \phi^{n+1}_{h}) \rvert &\leq \sqrt{d} \| p^{n+1}-q^{n+1}_{h} \| \| \nabla \phi^{n+1}_{h} \| \leq \frac{d C_{r}}{\epsilon_{14}} \| p^{n+1}-q^{n+1}_{h} \|^{2} + \frac{\epsilon_{14}}{r}\| \nabla \phi^{n+1}_{h} \|^{2},
\\ \lvert PrRa(\gamma \zeta^{n+1},\phi^{n+1}_{h})\rvert &\leq \frac{Pr^{2} Ra^{2} C^{2}_{PF,1} C^{2}_{PF,2} C_{r}}{\epsilon_{15}} \| \nabla \zeta^{n+1} \|^{2} + \frac{\epsilon_{15}}{r} \| \nabla \phi^{n+1}_{h} \|^{2},
\\ \lvert -PrRa(\gamma \psi^{n+1}_{h},\phi^{n+1}_{h})\rvert &\leq \frac{Pr^{2} Ra^{2} C^{2}_{PF,1} C^{2}_{PF,2} C_{r}}{\epsilon_{16}} \| \nabla \psi^{n+1}_{h} \|^{2} + \frac{\epsilon_{16}}{r} \| \nabla \phi^{n+1}_{h} \|^{2}.
\end{align}
Applying the estimates and Lemma \ref{consistency} into the temperature and velocity error equations (\ref{fet1}), (\ref{feu1}) and multiplying by $\Delta t$:
\begin{align}\label{error:thick:paramT}
&\frac{1}{2} \Big\{\|\psi^{n+1}_{h}\|^{2} - \|\psi^{n}_{h}\|^{2} + \|\psi^{n+1}_{h} - \psi^{n}_{h}\|^{2}\Big\} + \kappa \Delta t \|\nabla \psi^{n+1}_{h}\|^{2}
\\ &\leq \frac{\Delta t C_{r}C^{2}_{PF,2}}{\epsilon_1}\| \zeta_{t} \|^{2}_{L^{\infty}(t^n,t^{n+1};L^{2}(\Omega))} + \frac{\Delta t \epsilon_1}{r} \| \nabla \psi^{n+1}_{h} \|^{2} + \frac{C_{r} \kappa^{2} \Delta t}{\epsilon_2}\| \nabla \zeta^{n+1} \|^{2} + \frac{\Delta t \epsilon_2}{r} \| \nabla \psi^{n+1}_{h} \|^{2} \notag
\\ &+ \frac{C_{3}^{2} C_{r} \Delta t}{\epsilon_3}\| \nabla u^{n+1} \|^{2} \| \nabla \zeta^{n+1} \|^{2} + \frac{\Delta t \epsilon_3}{r} \| \nabla \psi^{n+1}_{h} \|^{2} + \frac{C_{r} C_{3}^{2} \Delta t^{3}}{\epsilon_4}\| \nabla T^{n+1}_{h} \|^{2} \| \nabla u_{t} \|^{2}_{L^{\infty}(t^n,t^{n+1};L^{2}(\Omega))} \notag
\\ &+ \frac{\Delta t \epsilon_4}{r} \| \nabla \psi^{n+1}_{h} \|^{2} + \frac{C_{r}C_{3}^{2}\Delta t}{\epsilon_5}\| \nabla \eta^{n} \|^{2} \| \nabla T^{n+1}_{h} \|^{2} + \frac{\Delta t \epsilon_5}{r} \| \nabla \psi^{n+1}_{h} \|^{2} + \frac{C_{4} C_{T} \Delta t \epsilon_6}{2}\| \nabla \psi^{n+1}_{h} \|^{2} \notag
\\ &+ \frac{C_{4}C_{T}\Delta t \delta_6}{4 \epsilon_6}\| \nabla \phi^{n}_{h} \|^{2} + \frac{C_{4} C_{T} \Delta t}{4\epsilon_6 \delta_6} \| \phi^{n}_{h} \|^{2} + \frac{C_{r} C_{3}^{2} \Delta t}{\epsilon_7}\| \nabla {u'}^{n}_{h} \|^{2} \| \nabla \zeta^{n+1} \|^{2} + \frac{\Delta t \epsilon_7}{r} \| \nabla \psi^{n+1}_{h} \|^{2} \notag
\\ &+ \frac{C_{r} C_{3}^{2} \Delta t}{\epsilon_8} \| \nabla {u'}^{n}_{h} \|^{2} \| \nabla \zeta^{n} \|^{2}  + \frac{\Delta t \epsilon_8}{r} \| \nabla \psi^{n+1}_{h} \|^{2} + \frac{C_{6}^{2} C_{inv,2} \Delta t^{2}}{h \epsilon_9} \| \nabla {u'}^{n}_{h} \|^{2} \| \nabla \psi^{n+1}_{h} \|^{2} + \frac{\epsilon_9}{2} \| \psi^{n+1}_{h} - \psi^{n}_{h} \|^{2} \notag
\\ &+ \frac{C_{r}C_{3}^{2} \Delta t^{3}}{\epsilon_{10}}\| \nabla {u'}^{n}_{h} \|^{2} \| \nabla T_{t} \|^{2}_{L^{\infty}(t^n,t^{n+1};L^{2}(\Omega))} + \frac{\Delta t \epsilon_{10}}{r} \| \nabla \psi^{n+1}_{h} \|^{2} \notag
\\ &+ \frac{C^{2}_{PF,2} C_{r} \Delta t^{3}}{\epsilon_{11}} \| T_{tt}\|^{2}_{L^{\infty}(t^n,t^{n+1};L^{2}(\Omega))} + \frac{\epsilon_{11}}{r} \| \nabla \psi^{n+1}_{h} \|^{2}, \notag
\end{align}
\noindent and
\begin{align}\label{error:thick:paramU}
&\frac{1}{2} \Big\{\|\phi^{n+1}_{h}\|^{2} - \|\phi^{n}_{h}\|^{2} + \|\phi^{n+1}_{h} - \phi^{n}_{h}\|^{2}\Big\} + Pr \Delta t \|\nabla \phi^{n+1}_{h}\|^{2}
\\ &\leq \frac{\Delta t C_{r}C^{2}_{PF,1}}{\epsilon_{12}} \| \eta_{t} \|^{2}_{L^{\infty}(t^n,t^{n+1};L^{2}(\Omega))} + \frac{\Delta t \epsilon_{12}}{r} \| \nabla \phi^{n+1}_{h} \|^{2} + \frac{C_{r} Pr^{2} \Delta t}{\epsilon_{13}}\| \nabla \eta^{n+1} \|^{2} \notag
\\ &+ \frac{\Delta t \epsilon_{13}}{r} \| \nabla \phi^{n+1}_{h} \|^{2} + \frac{dC_{r}\Delta t}{\epsilon_{14}}\| p^{n+1} - q^{n+1}_{h} \|^{2} + \frac{\Delta \epsilon_{14}}{r} \| \nabla \phi^{n+1}_{h} \|^{2} \notag
\\ &+ \Delta t Pr^{2} Ra^{2} C^{2}_{PF,1} C^{2}_{PF,2} C_{r} \big( \frac{1}{\epsilon_{15}}\| \nabla \zeta^{n+1} \|^{2} + \frac{1}{\epsilon_{16}}\| \nabla \psi^{n+1}_{h} \|^{2} \big) + \frac{\Delta t}{r} \big( \frac{1}{\epsilon_{15}} \| \nabla \phi^{n+1}_{h} \|^{2} + \frac{1}{\epsilon_{16}} \| \nabla \phi^{n+1}_{h} \|^{2} \big) \notag
\\ &+ \frac{C_{1} C_{r} \Delta t}{\epsilon_{17}}\| \nabla u^{n+1} \|^{2} \| \nabla \eta^{n+1} \|^{2} + \frac{\Delta t \epsilon_{17}}{r} \| \nabla \phi^{n+1}_{h} \|^{2} \notag
\\ &+ \frac{C_{r}C_{1}^{2} \Delta t^{3}}{\epsilon_{18}} \| \nabla u^{n+1}_{h} \|^{2} \| \nabla u_{t} \|^{2}_{L^{\infty}(t^n,t^{n+1};L^{2}(\Omega))} + \frac{\Delta t \epsilon_{18}}{r} \| \nabla \phi^{n+1}_{h} \|^{2} + \frac{C_{r}C_{1}\Delta t}{\epsilon_{19}}\| \nabla \eta^{n} \|^{2} \| \nabla u^{n+1}_{h} \|^{2} \notag
\\ &+ \frac{\Delta t \epsilon_{19}}{r} \| \nabla \phi^{n+1}_{h} \|^{2} + \frac{C_{2} C_{u} \Delta t \epsilon_{20}}{2}\| \nabla \phi^{n+1}_{h} \|^{2} + \frac{C_{2}C_{u}\Delta t \delta_{20}}{4 \epsilon_{20}}\| \nabla \phi^{n}_{h} \|^{2} + \frac{C_{2} C_{u} \Delta t}{4\epsilon_{20} \delta_{20}} \| \phi^{n}_{h} \|^{2} \notag
\\ &+ \frac{C_{r} C_{1}^{2} \Delta t}{\epsilon_{21}}\| \nabla {u'}^{n}_{h} \|^{2} \| \nabla \eta^{n+1} \|^{2} + \frac{\Delta t \epsilon_{21}}{r} \| \nabla \phi^{n+1}_{h} \|^{2} + \frac{C_{r} C_{1} \Delta t}{\epsilon_{22}} \| \nabla {u'}^{n}_{h} \|^{2} \| \nabla \eta^{n} \|^{2} + \frac{\Delta t \epsilon_{22}}{r} \| \nabla \phi^{n+1}_{h} \|^{2} \notag
\\ &+ \frac{C_{5}^{2} C_{inv,1} \Delta t^{2}}{h \epsilon_{23}} \| \nabla {u'}^{n}_{h} \|^{2} \| \nabla \phi^{n+1}_{h} \|^{2} + \frac{\epsilon_{23}}{2} \| \phi^{n+1}_{h} - \phi^{n}_{h} \|^{2} + \frac{C_{r}C_{1} \Delta t^{3}}{\epsilon_{24}}\| \nabla {u'}^{n}_{h} \|^{2} \| \nabla u_{t} \|^{2}_{L^{\infty}(t^n,t^{n+1};L^{2}(\Omega))} \notag
\\ &+ \frac{\Delta t \epsilon_{24}}{r} \| \nabla \phi^{n+1}_{h} \|^{2} + \frac{C^{2}_{PF,1}C_{r} \Delta t^{3}}{\epsilon_{26}} \| u_{tt} \|^{2}_{L^{\infty}(t^n,t^{n+1};L^{2}(\Omega))} + \frac{\Delta \epsilon_{26}}{r} \| \nabla \phi^{n+1}_{h} \|^{2}. \notag
\end{align}
Combine (\ref{error:thick:paramT}) and (\ref{error:thick:paramU}), choose free parameters appropriately, use condition (\ref{c1}), and take the maximum over all constants on the r.h.s.  Then,
\begin{align} 
&\frac{1}{2} \big( \|\psi^{n+1}_{h}\|^{2} - \|\psi^{n}_{h}\|^{2} \big) + \frac{1}{4}\|\psi^{n+1}_{h} - \psi^{n}_{h}\|^{2} + \frac{\kappa \Delta t}{4} \big(\|\nabla \psi^{n+1}_{h}\|^{2} - \|\nabla \psi^{n}_{h}\|^{2}\big)
\\ &+ \frac{1}{2} \big( \|\phi^{n+1}_{h}\|^{2} - \|\phi^{n}_{h}\|^{2} \big) + \frac{1}{4} \|\phi^{n+1}_{h} - \phi^{n}_{h}\|^{2} + \frac{Pr \Delta t}{4} \big(\|\nabla \phi^{n+1}_{h}\|^{2} - \|\nabla \phi^{n}_{h}\|^{2} \big) \notag
\\ &\leq C\Big\{\Delta t \| \zeta_{t} \|^{2}_{L^{\infty}(t^n,t^{n+1};L^{2}(\Omega))} + \Delta t \| \nabla \zeta^{n+1} \|^{2} + \Delta t \| \nabla u^{n+1} \|^{2} \| \nabla \zeta^{n+1} \|^{2} \notag
\\ &+ \Delta t^{3} \| \nabla T^{n+1}_{h} \|^{2} \| \nabla u_{t} \|^{2}_{L^{\infty}(t^n,t^{n+1};L^{2}(\Omega))} + \Delta t \| \nabla \eta^{n} \|^{2} \| \nabla T^{n+1}_{h} \|^{2} + \Delta t \| \phi^{n}_{h} \|^{2} + \Delta t \| \nabla {u'}^{n}_{h} \|^{2} \| \nabla \zeta^{n+1} \|^{2} \notag
\\ &+ \Delta t \| \nabla {u'}^{n}_{h} \|^{2} \| \nabla \zeta^{n} \|^{2} + \Delta t^{3} \| \nabla {u'}^{n}_{h} \|^{2} \| \nabla T_{t} \|^{2}_{L^{\infty}(t^n,t^{n+1};L^{2}(\Omega))} + \Delta t^{3} \| T_{tt}\|^{2}_{L^{\infty}(t^n,t^{n+1};L^{2}(\Omega))} \notag
\\ &+ \Delta t \| p^{n+1} - q^{n+1}_{h} \|^{2} + \Delta t\| \eta_{t} \|^{2}_{L^{\infty}(t^n,t^{n+1};L^{2}(\Omega))} + \Delta t \| \nabla \eta^{n+1} \|^{2} + \Delta t \| \nabla \zeta^{n+1} \|^{2} \notag
\\ &+ \Delta t \| \nabla u^{n+1} \|^{2} \| \nabla \eta^{n+1} \|^{2} + \Delta t^{3} \| \nabla u^{n+1}_{h} \|^{2} \| \nabla u_{t} \|^{2}_{L^{\infty}(t^n,t^{n+1};L^{2}(\Omega))} + \Delta t \| \nabla \eta^{n} \|^{2} \| \nabla u^{n+1}_{h} \|^{2} \notag
\\ &+ \Delta t \| \nabla {u'}^{n}_{h} \|^{2} \| \nabla \eta^{n+1} \|^{2} + \Delta t \| \nabla {u'}^{n}_{h} \|^{2} \| \nabla \eta^{n} \|^{2} + \Delta t^{3} \| \nabla {u'}^{n}_{h} \|^{2} \| \nabla u_{t} \|^{2}_{L^{\infty}(t^n,t^{n+1};L^{2}(\Omega))} \notag
\\ &+ \Delta t^{3} \| u_{tt} \|^{2}_{L^{\infty}(t^n,t^{n+1};L^{2}(\Omega))} \Big\}. \notag
\end{align}
Multiply by 2, sum from $n = 0$ to $n = N-1$, apply Lemma \ref{l4}, and renorm.  Then,
\begin{align*}
&\|\psi^{N}_{h}\|^{2} + \|\phi^{N}_{h}\|^{2} + \frac{1}{2}\sum_{n = 0}^{N-1}\big(\|\psi^{n+1}_{h} - \psi^{n}_{h}\|^{2} + \|\phi^{n+1}_{h} - \phi^{n}_{h}\|^{2}\big) + \frac{\kappa \Delta t}{2}\|\nabla \psi^{N}_{h}\|^{2} + \frac{Pr \Delta t}{2} \|\nabla \phi^{N}_{h}\|^{2}
\\ &\leq  C\Big\{(2 + \| \nabla u^{n+1} \|^{2} + 2\|\nabla {u'}^{n}_{h}\|^{2}) \Delta t \vertiii{\nabla \zeta}^{2}_{\infty,0} + (1 + \| \nabla T^{n+1}_{h} \|^{2} + \| \nabla u^{n+1} \|^{2} + \|\nabla u^{n+1}_{h}\|^{2} + 2\|\nabla {u'}^{n}_{h}\|^{2}) \Delta t \vertiii{\nabla \eta}^{2}_{\infty,0} 
\\ &+ \Delta t \vertiii{\zeta_{t}}^{2}_{\infty,0} + \big(\| \nabla T^{n+1}_{h} \|^{2} + \| \nabla u^{n+1}_{h} \|^{2} + \| \nabla {u'}^{n}_{h} \|^{2} \big) \Delta t^{3} \vertiii{\nabla u_{t}}^{2}_{\infty,0} + \Delta t^{3} \| \nabla {u'}^{n}_{h} \|^{2} \vertiii{\nabla T_{t}}^{2}_{\infty,0} + \Delta t^{3} \vertiii{ T_{tt} }^{2}_{2,0}
\\ &+ \Delta t \vertiii{ p - q_{h} }^{2}_{\infty,0} + \Delta t \vertiii{ \eta_{t} }^{2}_{\infty,0} + \Delta t^{3} \vertiii{ u_{tt} }^{2}_{\infty,0}\Big\} + \|\psi^{0}_{h}\|^{2} + \frac{\kappa \Delta t}{2}\|\nabla \psi^{0}_{h}\|^{2} + \|\phi^{0}_{h}\|^{2} + \frac{Pr \Delta t}{2}\|\nabla \phi^{0}_{h}\|^{2}
\\ &\leq  C\Big\{\Delta t \vertiii{\nabla \zeta}^{2}_{\infty,0} + \Delta t \vertiii{\nabla \eta}^{2}_{\infty,0} + \Delta t \vertiii{\zeta_{t}}^{2}_{\infty,0} + \Delta t^3 \vertiii{\nabla u_{t}}^{2}_{\infty,0} + \Delta t^3 \vertiii{\nabla T_{t}}^{2}_{\infty,0}
\\ &+ \Delta t \vertiii{ p - q_{h} }^{2}_{\infty,0} + \Delta t \vertiii{ \eta_{t} }^{2}_{\infty,0} + \Delta t^{3} \vertiii{ u_{tt} }^{2}_{\infty,0}\Big\} + \|\psi^{0}_{h}\|^{2} + \frac{\kappa \Delta t}{2}\|\nabla \psi^{0}_{h}\|^{2} + \|\phi^{0}_{h}\|^{2} + \frac{Pr \Delta t}{2}\|\nabla \phi^{0}_{h}\|^{2}.
\end{align*}
Take infimums over $X_{h}$, $Q_{h}$, and $W_{h}$.  Apply the triangle inequality, then
\begin{multline*}
\|e^{N}_{T}\|^{2} + \|e^{N}_{u}\|^{2} + \frac{1}{2}\sum_{n = 0}^{N-1}\big(\|e^{n+1}_{T} - e^{n}_{T}\|^{2} + \|e^{n+1}_{u} - e^{n}_{u}\|^{2}\big) + \frac{\kappa \Delta t}{2}\|\nabla e^{N}_{T}\|^{2} + \frac{Pr \Delta t}{2} \|\nabla e^{N}_{u}\|^{2}
\\ \leq  C \Big\{ \Delta t \inf_{v_{h} \in X_{h}} \Big( \vertiii{\nabla (u - v_{h}) }^{2}_{\infty,0} + \vertiii{(u - v_{h})_{t}}^{2}_{\infty,0} \Big) + \Delta t \inf_{S_{h} \in W_{h}} \Big( \vertiii{\nabla (T - S_{h})}^{2}_{\infty,0} + \vertiii{(T - S_{h})_{t}}^{2}_{\infty,0} \Big)
\\ + \Delta t \inf_{q_{h} \in Q_{h}} \vertiii{ p - q_{h} }^{2}_{\infty,0} + \Delta t^{3}  + \Delta t \|\nabla \eta^{0}\|^{2} + \Delta t \|\nabla \zeta^{0}\|^{2} + \|\eta^{0}\|^{2} 
\\ + \|\zeta^{0}\|^{2} + \| e^{0}_{T} \|^{2} + \| e^{0}_{u} \|^{2} + \Delta t \| \nabla e^{0}_{T} \|^{2} + \Delta t\| \nabla e^{0}_{u} \|^{2}\Big\}.
\end{multline*}
\end{proof}
The same result holds, with a different constant, for the thin wall problem.
\begin{theorem} \label{error:thin}
For (u,p,T) satisfying (9) - (13), suppose that $(u^{0}_{h},p^{0}_{h},T^{0}_{h}) \in (X_{h},Q_{h},W_{h})$ are approximations of $(u^{0},p^{0},T^{0})$ to within the accuracy of the interpolant.  Further, suppose that condition (\ref{c1}) holds. Then there exists a constant C such that
\begin{multline*}
\|e^{N}_{T}\|^{2} + \|e^{N}_{u}\|^{2} + \frac{1}{2}\sum_{n = 0}^{N-1}\big(\|e^{n+1}_{T} - e^{n}_{T}\|^{2} + \|e^{n+1}_{u} - e^{n}_{u}\|^{2}\big) + \frac{\kappa \Delta t}{2}\|\nabla e^{N}_{T}\|^{2} + \frac{Pr \Delta t}{2} \|\nabla e^{N}_{u}\|^{2}
\\ \leq  C \Big\{ \Delta t \inf_{v_{h} \in X_{h}} \Big( \vertiii{\nabla (u - v_{h}) }^{2}_{\infty,0} + \vertiii{(u - v_{h})_{t}}^{2}_{\infty,0} \Big) + \Delta t \inf_{S_{h} \in W_{h}} \Big( \vertiii{\nabla (T - S_{h})}^{2}_{\infty,0} + \vertiii{(T - S_{h})_{t}}^{2}_{\infty,0} \Big)
\\ + \Delta t \inf_{q_{h} \in Q_{h}} \vertiii{ p - q_{h} }^{2}_{\infty,0} + \Delta t^{3}  + \Delta t \|\nabla \eta^{0}\|^{2} + \Delta t \|\nabla \zeta^{0}\|^{2} + \|\eta^{0}\|^{2} 
\\ + \|\zeta^{0}\|^{2} + \| e^{0}_{T} \|^{2} + \| e^{0}_{u} \|^{2} + \Delta t \| \nabla e^{0}_{T} \|^{2} + \Delta t\| \nabla e^{0}_{u} \|^{2}\Big\}.
\end{multline*}
\end{theorem}
\begin{proof}
We follow the same methodology as in Theorem \ref{error:thick}.  The error equations for velocity and temperature are
\begin{multline}
(\frac{e^{n+1}_{u} - e^{n}_{u}}{\Delta t},v_{h}) - b(u^{n}_{h}-{u'}^{n}_{h},u^{n+1}_{h},v_{h}) - b({u'}^{n}_{h},u^{n}_{h},v_{h}) + Pr (\nabla e^{n+1}_{u},\nabla v_{h}) - (e^{n+1}_{p}, \nabla \cdot v_{h}) \label{error:two:veleqn}
\\ = Pr Ra \Big\{ (\gamma T^{n+1},v_{h}) - (\gamma T^{n}_{h},v_{h}) \Big\} + \tau_{u}(u^{n+1},v_{h}) \; \; \forall v_{h} \in X_{h},
\end{multline}
\begin{multline}
(\frac{e^{n+1}_{T} - e^{n}_{T}}{\Delta t},S_{h}) + b^{\ast}(u^{n+1},T^{n+1},S_{h}) - b^{\ast}(u^{n}_{h}-{u'}^{n}_{h},T^{n}_{h},S_{h}) - b^{\ast}({u'}^{n}_{h},T^{n}_{h},S_{h}) + \kappa (\nabla e^{n+1}_{T},\nabla S_{h})  \label{error:two:tempeqn}
\\ = (u^{n+1}_{1},S_{h}) - (u^{n}_{1h},S_{h}) + \tau_{T}(T^{n+1},S_{h}) \; \; \forall S_{h} \in W_{h}.
\end{multline}
Add and subtract $Pr Ra (\gamma T^{n},v_{h})$ in (\ref{error:two:veleqn}) and $(u^{n}_{1},S_{h})$ in (\ref{error:two:tempeqn}).  Then,
\begin{multline}
(\frac{e^{n+1}_{u} - e^{n}_{u}}{\Delta t},v_{h}) - b(u^{n}_{h}-{u'}^{n}_{h},u^{n}_{h},v_{h}) - b({u'}^{n}_{h},u^{n}_{h},v_{h}) + Pr (\nabla e^{n+1}_{u},\nabla v_{h}) - (e^{n+1}_{p}, \nabla \cdot v_{h}) 
\\ = Pr Ra \Big\{ (\gamma (T^{n+1}-T^{n}),v_{h}) - (\gamma e^{n}_{T},v_{h}) \Big\} + \tau_{u}(u^{n+1},v_{h}) \; \; \forall v_{h} \in X_{h}, 
\end{multline}
\begin{multline}
(\frac{e^{n+1}_{T} - e^{n}_{T}}{\Delta t},S_{h}) + b^{\ast}(u^{n+1},T^{n+1},S_{h}) - b^{\ast}(u^{n}_{h}-{u'}^{n}_{h},T^{n}_{h},S_{h}) - b^{\ast}({u'}^{n}_{h},T^{n}_{h},S_{h}) + \kappa (\nabla e^{n+1}_{T},\nabla S_{h}) 
\\ = (u^{n+1}_{1} - u^{n}_{1},S_{h}) - (e^{n}_{u1},S_{h}) + \tau_{T}(T^{n+1},S_{h}) \; \; \forall S_{h} \in W_{h}.
\end{multline}
Estimate the new terms using similar techniques as in Theorem \ref{error:thick}:
\begin{align}
|Pr Ra (\gamma (T^{n+1}-T^{n}),v_{h})| &\leq \frac{Pr^{2}Ra^{2}C_{PF,1}^{2}C_{r}}{\epsilon_{26}} \|T^{n+1} - T^{n} \|^{2} + \frac{\epsilon_{26}}{r} \| \nabla v_{h}\|^{2}
\\ &\leq \frac{Pr^{2}Ra^{2}C_{PF,1}^{2}C_{r} \Delta t^{2}}{\epsilon_{26}} \| T_{t} \|^{2}_{L^{\infty}(t^n,t^{n+1};L^{2}(\Omega))} + \frac{\epsilon_{26}}{r} \| \nabla v_{h}\|^{2},\notag
\\ |Pr Ra (\gamma e^{n}_{T},v_{h})| &= |Pr Ra (\gamma \zeta^{n}, v_{h}) - Pr Ra (\gamma \psi^{n}_{h},v_{h})| 
\\ &\leq \frac{Pr^{2}Ra^{2}C_{PF,1}^{2}C_{r}}{\epsilon_{27}} (\| \zeta^{n} \|^{2} + \| \psi^{n}_{h} \|^{2}) + \frac{2 \epsilon_{27}}{r} \| \nabla v_{h}\|^{2},\notag
\\ |(u^{n+1}_{1}-u^{n}_{1},S_{h})| &\leq \frac{C_{PF,2}^{2} C_{r}}{\epsilon_{28}} \| u^{n+1}_{1}-u^{n}_{1} \|^{2} + \frac{\epsilon}{r} \| \nabla S_{h}\|^{2}
\\ &\leq \frac{C_{PF,2}^{2} C_{r} \Delta t^{2}}{\epsilon_{28}} \| u_{t} \|^{2}_{L^{\infty}(t^n,t^{n+1};L^{2}(\Omega))} + \frac{\epsilon_{28}}{r} \| \nabla S_{h}\|^{2},\notag
\\ |(e^{n}_{u_1},S_{h})| = |(\eta^{n}_{1},S_{h}) - (\phi^{n}_{1h},S_{h})| &\leq \frac{Pr^{2}Ra^{2}C_{PF,1}^{2}C_{r}}{\epsilon_{29}} (\| \eta^{n} \|^{2} + \| \phi^{n}_{h} \|^{2}) + \frac{2 \epsilon_{29}}{r} \| \nabla S_{h}\|^{2}.
\end{align}
Apply estimates similar to those in Theorem \ref{error:thick} as well as the above estimates, multiply by $2\Delta t$, sum from $n = 0$ to $n = N-1$.  Further, apply Lemma \ref{l4}, triangle inequality and arrive at the result.
\end{proof}
\begin{corollary}
	Suppose the assumptions of Theorem \ref{t1} hold.  Further suppose that the finite element spaces ($X_{h}$,$Q_{h}$,$W_{h}$) are given by P2-P1-P2 (Taylor-Hood), then the errors in velocity and temperature satisfy
	\begin{multline*}
		\|e^{N}_{T}\|^{2} + \|e^{N}_{u}\|^{2} + \frac{1}{2}\sum_{n = 0}^{N-1}\big(\|e^{n+1}_{T} - e^{n}_{T}\|^{2} + \|e^{n+1}_{u} - e^{n}_{u}\|^{2}\big) + \frac{\kappa \Delta t}{2}\|\nabla e^{N}_{T}\|^{2} + \frac{Pr \Delta t}{2} \|\nabla e^{N}_{u}\|^{2}
		\\ \leq C (\Delta t h^{4} + \Delta t h^{6} + \Delta t^{3} + \Delta t \|\nabla \eta^{0}\|^{2} + \Delta t \|\nabla \zeta^{0}\|^{2} + \|\eta^{0}\|^{2}
		\\ + \|\zeta^{0}\|^{2} + \| e^{0}_{T} \|^{2} + \| e^{0}_{u} \|^{2} + \Delta t \| \nabla e^{0}_{T} \|^{2} + \Delta t\| \nabla e^{0}_{u} \|^{2}).
	\end{multline*}
\end{corollary}
\begin{corollary}
	Suppose the assumptions of Theorem \ref{t1} hold.  Further suppose that the finite element spaces ($X_{h}$,$Q_{h}$,$W_{h}$) are given by P1b-P1-P1b (MINI element), then the errors in velocity and temperature satisfy
	\begin{multline*}
		\|e^{N}_{T}\|^{2} + \|e^{N}_{u}\|^{2} + \frac{1}{2}\sum_{n = 0}^{N-1}\big(\|e^{n+1}_{T} - e^{n}_{T}\|^{2} + \|e^{n+1}_{u} - e^{n}_{u}\|^{2}\big) + \frac{\kappa \Delta t}{2}\|\nabla e^{N}_{T}\|^{2} + \frac{Pr \Delta t}{2} \|\nabla e^{N}_{u}\|^{2}
		\\ \leq C (\Delta t h^{2} + \Delta t h^{4} + \Delta t^{3} + \Delta t \|\nabla \eta^{0}\|^{2} + \Delta t \|\nabla \zeta^{0}\|^{2} + \|\eta^{0}\|^{2}
		\\ + \|\zeta^{0}\|^{2} + \| e^{0}_{T} \|^{2} + \| e^{0}_{u} \|^{2} + \Delta t \| \nabla e^{0}_{T} \|^{2} + \Delta t\| \nabla e^{0}_{u} \|^{2}).
	\end{multline*}
\end{corollary}
\section{Numerical Experiments}
In this section, we illustrate the stability and convergence of the numerical scheme described by (\ref{scheme:two:velocity}) - (\ref{scheme:two:temperature}) using Taylor-Hood (P2-P1-P2) elements to approximate the average velocity, pressure, and temperature.  The numerical experiments include the double pane window benchmark problem of de Vahl Davis \cite{Davis}, a convergence experiment and predictability exploration with an analytical solution adopted from \cite{Zhang} devised through the method of manufactured solutions.  The software used for all tests is \textsc{FreeFem}$++$ \cite{Hecht}.
\subsection{Stability condition}
The constant appearing in condition (\ref{c1}) is estimated by pre-computations for the double pane window problem appearing below.  We set $C_{\dagger} = 1$.  The first condition is used and checked at each iteration.  If violated, the timestep is halved and the iteration is repeated.  The timestep is never increased.  The condition is violated three times during the computation of the double pane window problem with $Ra = 10^{6}$ in Section 5.3.
\subsection{Perturbation generation}
The bred vector (BV) algorithm of Toth and Kalnay \cite{Toth} is used to generate perturbations in the double pane window problem and in exploring predictability.  The BV algorithm produces a perturbation with maximal separation rate.  We set $J = 2$ and $d = 2$ in all experiments.  An initial random positive/negative perturbation pair was generated $\pm \epsilon = \pm (\epsilon_{1},\epsilon_{2},\epsilon_{3})$ with $\epsilon_{i} \in (0,0.01) \; \forall i = 1,2,3$; that is, a pair of initial perturbations for each component of velocity and temperature.  Utilizing the scheme (\ref{scheme:two:velocity}) - (\ref{scheme:two:temperature}), denote the control and perturbed numerical approximations $\chi^{n}_{h}$ and $\chi^{n}_{p,h}$, respectively.  Then, a bred vector $bv(\chi;\epsilon_{i})$ is generated via:

\textbf{Step one:}  Given $\chi^{0}_{h}$ and $\epsilon_{i}$, put $\chi^{0}_{p,h} = \chi^{0}_{h} + \epsilon_{i}$.  Select time reinitialization interval $\delta t \geq \Delta t$ and let \indent $t^{k} = k \delta t$ with $0 \leq k \leq k^{\ast} \leq N$.

\indent \textbf{Step two:} Compute $\chi^{k}_{h}$ and $\chi^{k}_{p,h}$.  Calculate $bv(\chi^{k};\epsilon_{i}) = \frac{\epsilon_{i}}{\| \chi^{k}_{p,h} - \chi^{k}_{h} \|} (\chi^{k}_{p,h} - \chi^{k}_{h})$.

\indent \textbf{Step three:}  Put $\chi^{k}_{p,h} = \chi^{k}_{h} + bv(\chi^{k};\epsilon_{i}) $.

\indent \textbf{Step four:} Repeat \textbf{Step two} with $k = k + 1$.

\indent \textbf{Step five:} Put $bv(\chi;\epsilon_{i}) = bv(\chi^{k^{\ast}};\epsilon_{i})$.

\noindent A positive/negative perturbed initial condition pair is generated via $\chi_{\pm} = \chi^{0} + bv(\chi;\pm \epsilon_{i})$.  We let $\delta t = \Delta t = 0.001$ and $k^{\ast} = 5$.

\begin{figure}
	\centering
	\includegraphics[width=5.5in,height=\textheight, keepaspectratio]{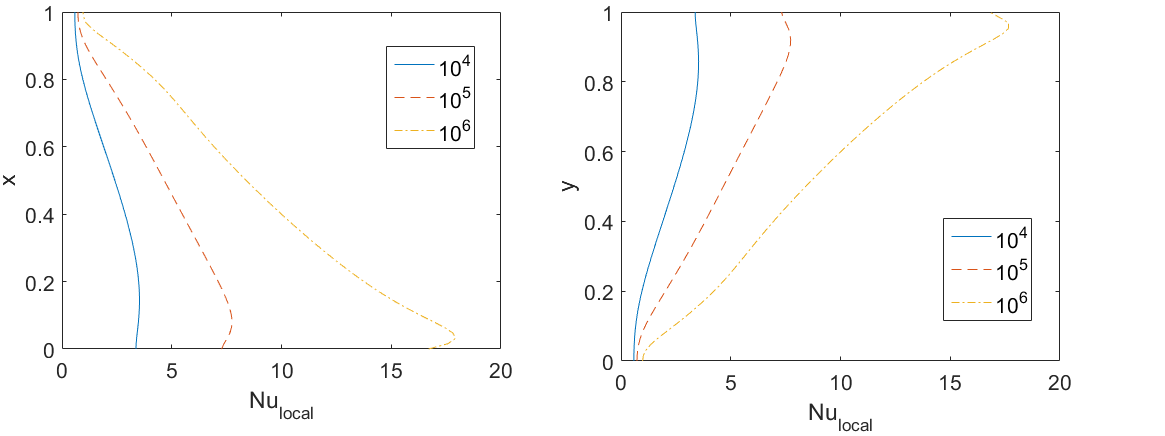}
	\caption{Variation of the local Nusselt number at the hot (left) and cold walls (right).}
\end{figure}

\subsection{The double pane window problem}
The first numerical experiment is the benchmark problem of de Vahl Davis \cite{Davis}.  The problem is the two-dimensional flow of a fluid in an unit square cavity with $Pr = 0.71$ and $\kappa = 1.0$. Both velocity components (i.e. $u = 0$) are zero on the boundaries. The horizontal walls are insulated and the left and right vertical walls are maintained at temperatures $T(0,y,t) = 1$ and $T(1,y,t) = 0$, respectively; see Figure 1b.  We let $10^3 \leq Ra \leq 10^6$.  The initial conditions for velocity and temperature are generated via the BV algorithm in Section 5.2,
\begin{align*}
u_{\pm}(x,y,0) := u(x,y,0;\omega_{1,2}) &= (1 + bv(u;\pm \epsilon_{1}), 1 + bv(u;\pm \epsilon_{2}))^{T}, \\
T_{\pm}(x,y,0) := T(x,y,0;\omega_{1,2}) &= 1 + bv(T;\pm \epsilon_{3}).
\end{align*}
Both $f(x,t;\omega_{j})$ and $g(x,t;\omega_{j})$ are identically zero for $j = 1,2$.  The finite element mesh is a division of $[0,1]^{2}$ into $64^{2}$ squares with diagonals connected with a line within each square in the same direction.  The stopping condition is
\begin{equation*}
\max_{0\leq n \leq N-1}\big\{\frac{\| u^{n+1}_{h} - u^{n}_{h}\|}{\| u^{n+1}_{h}\|},\frac{\|T^{n+1}_{h} - T^{n}_{h}\|}{\| T^{n+1}_{h}\|}{}\big\} \leq 10^{-5}
\end{equation*}
and initial timestep $\Delta t = 0.001$.  The timestep was halved three times to $0.000125$ to maintain stability for $Ra = 10^{6}$.  Several quantities are compared with benchmark solutions in the literature.  These include the maximum vertical velocity at $y = 0.5$, $\max_{x \in \Omega_{h}} {u_{2}(x,0.5,t^{\ast})}$, and maximum horizontal velocity at $x = 0.5$,  $\max_{y \in \Omega_{h}} {u_{1}(0.5,y,t^{\ast})}$.  We present our computed values for the mean flow in Tables 1 and 2 alongside several of those seen in the literature.  Furthermore, the local Nusselt number is calculated at the cold (+) and hot walls (-), respectively, via
\begin{equation*}
Nu_{local} =  \pm \frac{\partial{T}}{\partial{x}}.
\end{equation*}
The average Nusselt number on the vertical boundary at x = 0 is calculated via
\begin{equation*}
	Nu_{avg} =  \int^{1}_{0} Nu_{local} dy.
\end{equation*}
Figure 2 presents the plots of $Nu_{local}$ at the hot and cold walls.  Table 3 presents computed values of $Nu_{avg}$ alongside several of those seen in the literature.  Figures 3 and 4 present the velocity streamlines and temperature isotherms for the averages.  All results are seen to be in good agreement with the benchmark values in the literature \cite{Davis, Manzari, Wan, Cibik, Zhang}.

\vspace{5mm}
\begin{adjustbox}{max width=\textwidth}
	\begin{tabular}{ c  c  c  c  c  c  c  c }
		\hline			
		Ra & Present study & Ref. \cite{Davis} & Ref. \cite{Manzari} & Ref. \cite{Wan} & Ref. \cite{Cibik} & Ref. \cite{Zhang} \\
		\hline
		$10^{4}$ & 16.18 (64$\times$64) & 16.18 (41$\times$41) & 16.10 (71$\times$71) & 16.10 (101$\times$101) & 15.90 (11$\times$11) & 16.18 (64$\times$64)\\
		$10^{5}$ & 34.72 (64$\times$64) & 34.81 (81$\times$81) & 34 (71$\times$71) & 34 (101$\times$101) & 33.51 (21$\times$21) & 34.74 (64$\times$64) \\
		$10^{6}$ & 64.80 (64$\times$64) & 65.33 (81$\times$81) & 65.40 (71$\times$71) & 65.40 (101$\times$101) & 65.52 (32$\times$32) & 64.81 (64$\times$64)\\
		\hline  
	\end{tabular}
\end{adjustbox}
\captionof{table}{Comparison of maximum horizontal velocity at x = 0.5 together with mesh size used in computation for the double pane window problem.}
\begin{adjustbox}{max width=\textwidth}
	\begin{tabular}{ c  c  c  c  c  c  c  c }
		\hline			
		Ra & Present study & Ref. \cite{Davis} & Ref. \cite{Manzari} & Ref. \cite{Wan} & Ref. \cite{Cibik} & Ref. \cite{Zhang} \\
		\hline
		$10^{4}$ & 19.60 (64$\times$64) & 19.51 (41$\times$41) & 19.90 (71$\times$71) & 19.79 (101$\times$101) & 19.91 (11$\times$11) & 19.62 (64$\times$64)\\
		$10^{5}$ & 68.53 (64$\times$64) & 68.22 (81$\times$81) & 70 (71$\times$71) & 70.63 (101$\times$101) & 70.60 (21$\times$21) & 68.48 (64$\times$64) \\
		$10^{6}$ & 215.96 (64$\times$64) & 216.75 (81$\times$81) & 228 (71$\times$71) & 227.11 (101$\times$101) & 228.12 (32$\times$32) & 220.44 (64$\times$64)\\
		\hline  
	\end{tabular}
\end{adjustbox}
\captionof{table}{Comparison of maximum horizontal velocity at y = 0.5 together with mesh size used in computation for the double pane window problem.}
\centering
\begin{adjustbox}{max width=\textwidth}
	\begin{tabular}{ c  c  c  c  c  c  c  c }
		\hline			
		Ra & Present study & Ref. \cite{Davis} & Ref. \cite{Manzari} & Ref. \cite{Wan} & Ref. \cite{Cibik} & Ref. \cite{Zhang} \\
		\hline
		$10^{4}$ & 2.24 (64$\times$64) & 2.24 (41$\times$41) & 2.08 (71$\times$71) & 2.25 (101$\times$101) & 2.15 (11$\times$11) & 2.25 (64$\times$64)\\
		$10^{5}$ & 4.52 (64$\times$64) & 4.52 (81$\times$81) & 4.30 (71$\times$71) & 4.59 (101$\times$101) & 4.35 (21$\times$21) & 4.53 (64$\times$64) \\
		$10^{6}$ & 8.87 (64$\times$64) & 8.92 (81$\times$81) & 8.74 (71$\times$71) & 8.97 (101$\times$101) & 8.83 (32$\times$32) & 8.87 (64$\times$64)\\
		\hline  
	\end{tabular}
\end{adjustbox}
\captionof{table}{Comparison of average Nusselt number on the vertical boundary at x = 0 together with mesh size used in computation for the double pane window problem.}
\subsection{Numerical convergence study}
	In this section, we illustrate the convergence rates for the proposed algorithm (\ref{scheme:two:velocity}) - (\ref{scheme:two:temperature}).  The unperturbed solution is given by
\begin{align*}
u(x,y,t) &= (10x^2(x-1)^2y(y-1)(2y-1)\cos(t), -10x(x-1)(2x-1)y^2(y-1)^2\cos(t))^T, \\
T(x,y,t) &= u_{1}(x,y,t) + u_{2}(x,y,t), \\
p(x,y,t) &= 10(2x-1)(2y-1)\cos(t),
\end{align*}
with $\kappa = Pr = 1.0$, $Ra = 100$, and $\Omega = [0,1]^{2}$.  The perturbed solutions are given by
\begin{align*}
u(x,y,t;\omega_{1,2}) = (1 + \epsilon_{1,2})u(x,y,t), \\
T(x,y,t;\omega_{1,2}) = (1 + \epsilon_{1,2})T(x,y,t), \\
p(x,y,t;\omega_{1,2}) = (1 + \epsilon_{1,2})p(x,y,t), 
\end{align*} 
where $\epsilon_{1} = 1e-2 = -\epsilon_{2}$ and both forcing and boundary terms are adjusted appropriately.  The perturbed solutions satisfy the following relations,
\begin{align*}
< u > = 0.5\big(u(x,y,t;\omega_{1}) + u(x,y,t;\omega_{2}) \big) = u(x,y,t), \\
< T > = 0.5\big(T(x,y,t;\omega_{1}) + T(x,y,t;\omega_{2}) \big) = T(x,y,t), \\
< p > = 0.5\big(p(x,y,t;\omega_{1}) + p(x,y,t;\omega_{2}) \big) = p(x,y,t).
\end{align*}
The finite element mesh is a Delaunay triangulation generated from $m$ points on each side of $\Omega$.  We calculate errors in the approximations of the average velocity, temperature and pressure with the $L^{\infty}(0,t^{\ast};L^{2}(\Omega))$ and $L^{\infty}(0,t^{\ast};H^{1}(\Omega))$ norms.  Rates are calculated from the errors at two successive $m_{1,2}$ or $\Delta t_{1,2}$ via
\begin{align*}
\frac{\log_{2}(e_{\chi}(m_{1})/e_{\chi}(m_{2}))}{\log_{2}(m_{1}/m_{2})},\\
\frac{\log_{2}(e_{\chi}(\Delta t_{1})/e_{\chi}(\Delta t_{2}))}{\log_{2}(\Delta t_{1}/\Delta t_{2})},
\end{align*}
respectively, with $\chi = u, T, p$.  We first illustrate spatial convergence.  We isolate the spatial error by first choosing a fixed timestep $\Delta t = 0.0001$ and setting the final time $t^{\ast} = 0.001$.  The parameter $m$ is varied between 4, 8, 16, 32, 64, and 128.  Results are presented in Table 4.  Third order convergence is observed in velocity and temperature  and second order convergence in pressure in the $L^{\infty}(0,t^{\ast};L^{2}(\Omega))$ norm and second order convergence in velocity and temperature in the $L^{\infty}(0,t^{\ast};H^{1}(\Omega))$ norm.

Temporal convergence is illustrated by choosing a fixed $m = 64$ and setting the final time $t^{\ast} = 1$.  The timestep is varied between 4, 8, 16, 32, 64, 128.  Table 5 confirms first order convergence in velocity, temperature, and pressure in the $L^{\infty}(0,t^{\ast};L^{2}(\Omega))$ norm and in velocity and temperature in the $L^{\infty}(0,t^{\ast};H^{1}(\Omega))$ norm.

\vspace{5mm}
\begin{adjustbox}{max width=\textwidth}
\begin{tabular}{ c  c  c  c  c  c  c  c  c  c  c  c }
	\hline			
	$1/m$ & $\vertiii{ <u_{h}>- u }_{\infty,0}$ & Rate & $\vertiii{ \nabla <u_{h}> - \nabla u }_{\infty,0}$ & Rate & $\vertiii{ <T_{h}> - T }_{\infty,0}$ & Rate & $\vertiii{ \nabla <T_{h}> - \nabla T }_{\infty,0}$ & Rate & $\vertiii{ <p_{h}> - p }_{\infty,0}$ & Rate \\
	\hline
	4 & 0.00134087 & - & 0.0376324 & - & 2.49E-04 & - & 0.0100481 & - & 0.427751 & -\\
	8 & 3.68E-04 & 1.87 & 0.0162059 & 1.22 & 3.03E-05 & 3.04 & 0.00171527 & 2.55 & 0.0256596 & 4.06 \\
	16 & 5.56E-05 & 2.73 & 0.00443669 & 1.87 & 4.95E-06 & 2.61 & 4.82E-04 & 1.83 & 0.00482023 & 2.41 \\
	32 & 6.35E-06 & 3.13 & 9.80E-04 & 2.18 & 5.71E-07 & 3.12 & 1.07E-04 & 2.18 & 1.10E-03 & 2.13 \\
	64 & 8.67E-07 & 2.87 & 2.70E-04 & 1.86 & 8.13E-08 & 2.81 & 3.01E-05 & 1.82 & 2.70E-04 & 2.02 \\
	128 & 1.06E-07 & 3.04 & 6.63E-05 & 2.03 & 9.56E-09 & 3.09 & 7.08E-06 & 2.09 & 6.58E-05 & 2.04 \\
	\hline  
\end{tabular}
\end{adjustbox}
\captionof{table}{Errors and rates for average velocity, temperature, and pressure in corresponding norms.} 
\begin{adjustbox}{max width=\textwidth}
\begin{tabular}{ c  c  c  c  c  c  c  c  c  c  c  c }
		\hline			
		$1/\Delta t$ & $\vertiii{ <u_{h}>- u }_{\infty,0}$ & Rate & $\vertiii{ \nabla <u_{h}> - \nabla u }_{\infty,0}$ & Rate & $\vertiii{ <T_{h}> - T }_{\infty,0}$ & Rate & $\vertiii{ \nabla <T_{h}> - \nabla T }_{\infty,0}$ & Rate & $\vertiii{ <p_{h}> - p }_{\infty,0}$ & Rate \\
		\hline
%		1 & 0.0167364 & - & 0.125627 & - & 0.000182081 & - & 0.00128925 & - & 0.293405 & -\\
%		2 & 0.0121441 & 0.46 & 0.0911629 & 0.46 &1.36E-04 & 0.42 & 0.000965322 & 0.42 & 0.213008 & 0.46 \\
		4 & 0.00698068 & - & 0.0524076 & - & 1.12E-04 & - & 0.000798053 & - & 0.122182 & - \\
		8 & 0.0036989 & 0.92 & 0.0277725 & 0.92 & 6.78E-05 & 0.73 & 4.81E-04 & 0.73 & 0.0647005 & 0.92 \\
		16 & 0.001898 & 0.96 & 0.0142518 & 0.96 & 3.66E-05 & 0.89 & 2.60E-04 & 0.89 & 0.0331928 & 0.96 \\
		32 & 9.61E-04 & 0.98 & 0.00721454 & 0.98 & 1.89E-05 & 0.95 & 1.35E-04 & 0.94 & 0.0168049 & 0.98 \\
		64 & 4.83E-04 & 0.99 & 0.00363088 & 0.99 & 9.62E-06 & 0.98 & 7.02E-05 & 0.95 & 0.00846082 & 0.99 \\
		128 & 2.42E-04 & 1.00 & 0.00182511 & 0.99 & 4.85E-06 & 0.99 & 3.81E-05 & 0.89 & 0.0042531 & 0.99 \\
		\hline  
\end{tabular}
\end{adjustbox}
\captionof{table}{Errors and rates for average velocity, temperature, and pressure in corresponding norms.}
\subsection{Exploration of predictability}
Consider the problem with manufactured solution in Section 5.4.  However, instead of specifying the perturbations on the initial conditions, the BV algorithm in Section 5.2 yields
\begin{align*}
	u_{\pm}(x,y,0) := u(x,y,0;\omega_{1,2}) &= (u_{1}(x,y,0) + bv(u;\pm \epsilon_{1}), u_{2}(x,y,0) + bv(u;\pm \epsilon_{2}))^{T}, \\
	T_{\pm}(x,y,0) := T(x,y,0;\omega_{1,2}) &=  T(x,y,0) + bv(T;\pm \epsilon_{3}).
\end{align*}
The forcing functions and boundary conditions are left unperturbed.  Further, the Rayleigh number is varied between $10^2$ and $10^4$.  The initial timestep is 0.001 and final time $t^{\ast} = 0.5$.  Herein, we will define energy, variance, average effective Lyapunov exponent \cite{Boffetta}, and $\delta$-predictability horizon \cite{Boffetta}.
\begin{definition} The energy is given by
\begin{align*}
	Energy := \| T \| + \frac{1}{2}\| u \|^{2}.
\end{align*}
The variance of $\chi$ is
\begin{align*}
	V(\chi) := < \| \chi \|^{2} > - \| <\chi> \|^{2} = < \| \chi^{'} \|^{2} >.
\end{align*} 
The relative energy fluctuation is
\begin{align*}
	r(t) := \frac{\| \chi_{+} - \chi_{-} \|^{2}}{\| \chi_{+} \| \| \chi_{-} \|},
\end{align*}
and the \textit{average effective Lyapunov exponent} over $0 < \tau \leq t^{\ast}$ is
\begin{align*}
	\gamma_{\tau}(t) := \frac{1}{2 \tau}\log\big(\frac{r(t + \tau)}{r(t)}\big),
\end{align*}
with $0 < t + \tau \leq t^{\ast}$.  The $\delta$-\textit{predictability horizon} is
\begin{align*}
	t_{p} := \frac{1}{\gamma_{t^{\ast}}(0)}\log\Big(\frac{\delta}{\| (\chi_{+} - \chi_{-}) (0) \|}\Big).
\end{align*}
\end{definition}
Figure 3 presents the energy and variance of the approximate solutions with $Ra = 10^4$.  The variance of the perturbed solutions indicates that they do not deviate much from the mean and therefore not much from each other.  This seems to explain, in part, why the energy associated with these solutions is similar.  Interestingly, the energy associated with the unperturbed and mean computed solutions sit atop of one another; that is, the mean leads to a superior estimate than either member of the ensemble.  It seems that the BV algorithm generated a positive/negative initial condition pair leading to two solutions whose average approximates the unperturbed solution well.

Figures 4 and 5 present $\gamma_{t^{\ast}}(t)$ and $t_{p}$ for mean temperature and velocity approximations for $10^2 \leq Ra \leq 10^4$ and $\| (\chi_{+} - \chi_{-}) (0) \| \leq \delta \leq 0.15$.  The approximated effective Lyapunov exponent $\gamma_{t^{\ast}}(0)$ and $t_{p}$ are negative for both velocity and temperature for all Rayleigh numbers indicating a predictable flow.  However, $\gamma_{t^{\ast}}(t)$ changes sign for temperature and velocity at approximately $t = 0.11$ for $Ra = 10^4$ indicating a loss of predictability.
\begin{figure}
	\centering
	\includegraphics[width=5.5in,height=\textheight,keepaspectratio]{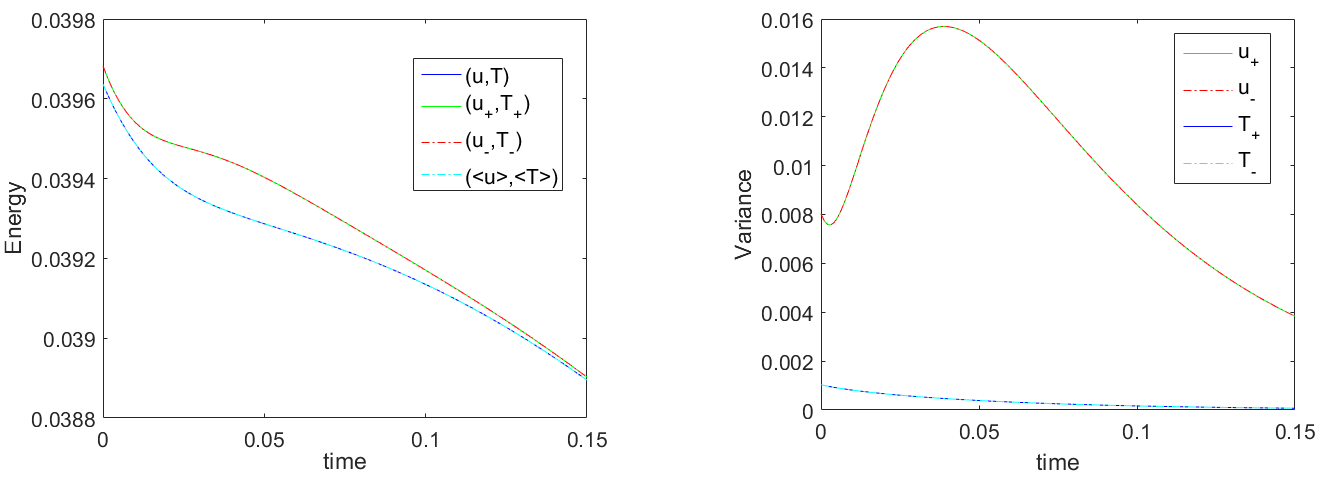}
	\caption{Comparison of the energy in the system (left) and variance of each velocity and temperature ensemble member (right).}
\end{figure}
\begin{figure}
	\centering
	\includegraphics[width=5.5in,height=\textheight,keepaspectratio]{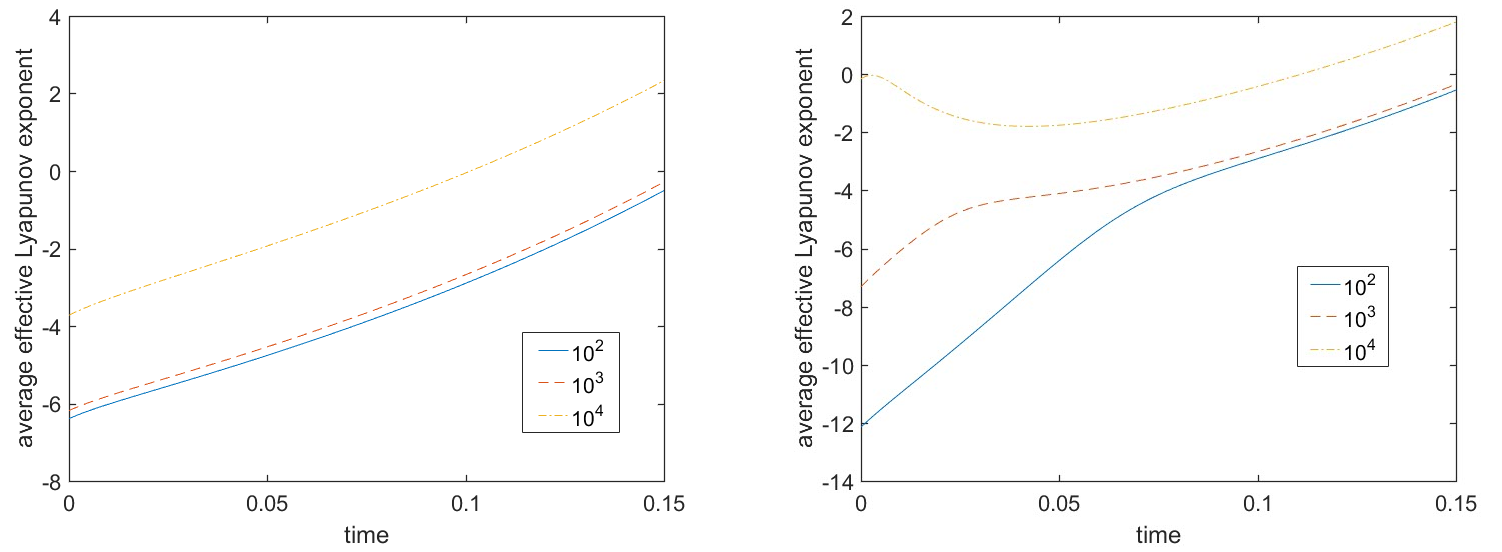}
	\caption{Comparison of average effective Lyapunov exponent for temperature (left) and velocity (right).}
\end{figure}
\begin{figure}
	\centering
	\includegraphics[width=5.5in,height=\textheight,keepaspectratio]{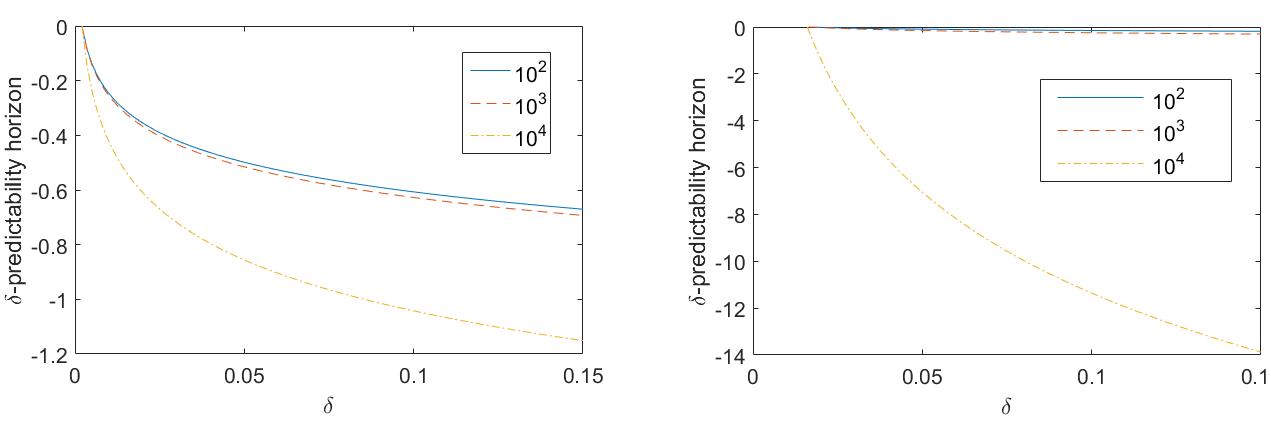}
	\caption{Comparison of $\delta$-predictability horizons for temperature (left) and velocity (right).}
\end{figure}

\section{Conclusion}
We presented two algorithms for calculating an ensemble of solutions to two laminar natural convection problems.  These algorithms addressed the competition between ensemble size and resolution in simulations.  In particular, both algorithms required the solution of a single matrix equation, at each time step, with multiple right hand sides.  Stability and convergence were proven and numerical experiments were performed to illustrate these properties.

\begin{figure}
	\includegraphics[width=\textwidth,height=\textheight,keepaspectratio]{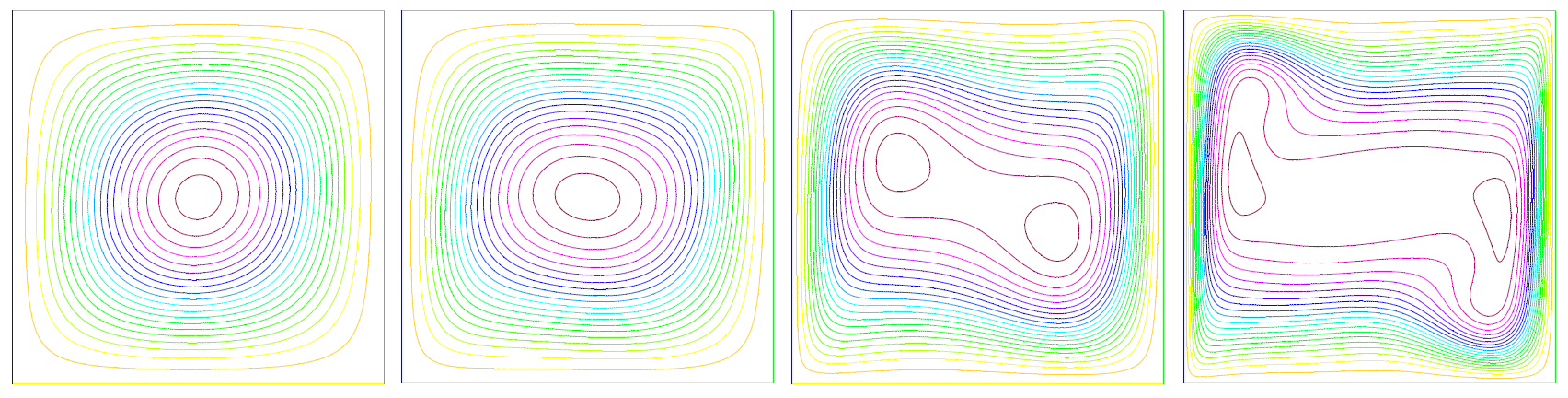}
	\caption{Streamlines for $Ra = 10^3, 10^4, 10^5,$ and $10^6$, from left to right, respectively.}
\end{figure}
\begin{figure}
	\includegraphics[width=\textwidth,height=\textheight,keepaspectratio]{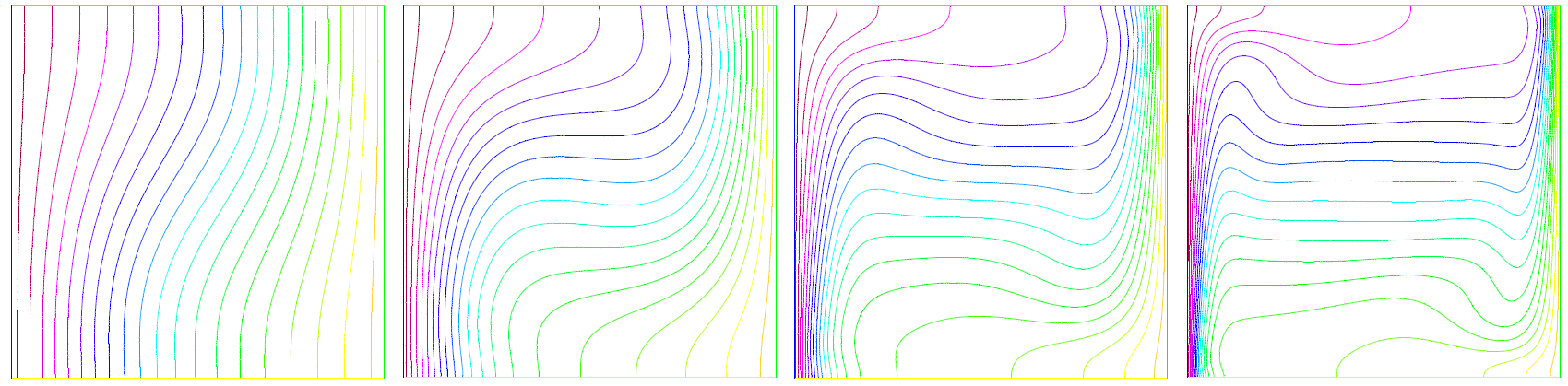}
	\caption{Isotherms for $Ra = 10^3, 10^4, 10^5,$ and $10^6$, from left to right, respectively.}
\end{figure}
% put your thanks here
\section*{Acknowledgments}
The author J.A.F. is supported by the DoD SMART Scholarship.  Moreover, the research herein was partially supported by NSF grants CBET 1609120 and DMS 1522267.  Further, the authors would like to thank Dr. Nan Jiang for her input and discussion.

\end{document}